\documentclass[10pt,fleqn]{scrartcl}

\usepackage[english]{babel}
\usepackage{amsmath,amsfonts, amsthm,xcolor,amssymb}
\usepackage{paralist}
\usepackage{authblk}
\numberwithin{equation}{section}
\usepackage{mathtools, mathabx}
\usepackage[colorlinks=true,urlcolor=blue, citecolor=blue,linkcolor=blue,linktocpage,pdfpagelabels, bookmarksnumbered,bookmarksopen]{hyperref}
\usepackage{graphicx}
\usepackage{comment}
\usepackage[hyperpageref]{backref}
\usepackage{amsmath, amssymb}
\usepackage{amsthm}
\usepackage{mathrsfs}
\usepackage{tikz} 
\usepackage{geometry}
\usepackage{wasysym}
\def\eps{\varepsilon}
\def\de{\partial}

\makeatletter
\renewcommand*{\@fnsymbol}[1]{\ifcase#1\or *\or **\or ***\or ****\else\@ctrerr\fi}
\makeatother

\theoremstyle{plain} 
\newtheorem{theorem}{Theorem}[section] 
\newtheorem{proposition}[theorem]{Proposition} 
\newtheorem{corollary}[theorem]{Corollary}     
\newtheorem{lemma}[theorem]{Lemma}             

\theoremstyle{definition}
\newtheorem{definition}[theorem]{Definition}   
\newtheorem{example}[theorem]{Example}         
\theoremstyle{remark}
\newtheorem{remark}[theorem]{Remark}

\numberwithin{equation}{section}
\font\manual=manfnt
\newcommand\xqed[1]{%
	\leavevmode\unskip\penalty9999 \hbox{}\nobreak\hfill
	\quad\hbox{#1}}
\newcommand\triang{\xqed{\manual\char'170}}
\usepackage{hyperref}

\begin{document}

\title{Optimal insulation and concentration breaking for nonlinear Robin boundary value problems}
\author{
Francesco Della Pietra\thanks{Universit\`a degli studi di Napoli Federico II, Dipartimento di Matematica e Applicazioni ``R. Caccioppoli'', Via Cintia, Monte S. Angelo - 80126 Napoli, Italia. Email:
f.dellapietra@unina.it} \and
Francescantonio Oliva\thanks{``Sapienza'' Universit\`a di Roma, Dipartimento di Scienze di Base e Applicate per l'Ingegneria, Via Scarpa 16, 00161 Roma, Italia. Email: francescantonio.oliva@uniroma1.it}
}
 
\date{}

\maketitle

\begin{abstract}
\noindent \textbf{Abstract}. We consider an optimal insulation problem for a bounded domain in $\mathbb{R}^N$ driven by the $p$-Laplace operator ($p>1$). We model the convective heat transfer between the body and the environment, which corresponds, before insulation, to a nonlinear Robin boundary value problem. Assuming the body is surrounded by a thin layer of insulating material of size $\varepsilon^{\frac{1}{p-1}}$, we compute the $\Gamma$-limit of the governing energy functional as $\varepsilon \to 0^+$. Furthermore, we study the optimization of the heat content among all possible distributions of the insulating material with a fixed total mass.
Finally, we highlight a concentration breaking phenomenon. Under a suitable non-degeneracy condition, if the boundary of the domain is connected or the external temperature profile is constant, the optimal insulating layer fails to cover the entire boundary whenever the total mass is sufficiently small. This is shown to be optimal: an explicit example provides that a disconnected boundary can trigger an anomalous double-phase transition, causing the insulation to fracture again even at intermediate mass regimes.

	
\end{abstract}

\begin{center}
	\begin{minipage}{.8\textwidth}
		\tableofcontents
	\end{minipage}
\end{center}

\section{Introduction}
Thermal insulation has gained a lot of interest in recent times as its importance in reducing energy consumption and mitigating climate change has been increasingly recognized. While physics and engineering predominantly explore the development of innovative materials and technological advancements to enhance thermal insulation, from the mathematical point of view an important area of research is given by shape optimization: the problem of determining the optimal geometric configuration of an insulating material to minimize heat transfer. 

The classical mathematical formulation of this problem focuses on the thermal insulation of a solid body subject to conductive heat transfer with its surroundings, typically modeled by Dirichlet boundary conditions (see for instance \cite{ab,cnt,Friedman80, HLL22, HLL24}). In this classical setting, a heated body represented by a bounded domain $\Omega \subset \mathbb{R}^N$ is surrounded by a thin layer of insulating material $\Sigma_\varepsilon$. The thickness of this layer is given by $\varepsilon h(\sigma)$, where $h$ is a profile function defined on the boundary $\partial\Omega$ and $\varepsilon$ is a small parameter. Assuming that the conductivity of the insulating layer is also of order $\varepsilon$, the physical phenomenon is governed by the minimization of the energy functional
\begin{equation*}
	G_{\varepsilon}(v,h) = \frac{1}{2}\int_{\Omega}|\nabla v|^{2}dx + \frac{\varepsilon}{2}\int_{\Sigma_{\varepsilon}}|\nabla v|^{2}dx - \int_{\Omega}fv~dx,
	\label{eq:G_eps}
\end{equation*}
where $v\in H_{0}^{1}(\Omega\cup\Sigma_{\varepsilon})$. This conventionally assumes that the external temperature is set to zero. Then the minimizer $u_\varepsilon$ satisfies a system of equations involving a transmission condition across the boundary $\partial\Omega$. As $\varepsilon \to 0$, $\Gamma$-convergence arguments show that the sequence of functionals $G_\varepsilon$ converges to a limit functional defined on $H^1(\Omega)$
\begin{equation*}
	G(v,h) = \frac{1}{2}\int_{\Omega}|\nabla v|^{2}dx + \frac{1}{2}\int_{\partial\Omega}\frac{v^{2}}{h}d\mathcal{H}^{N-1} - \int_{\Omega}fvdx,
	\label{eq:G_limit}
\end{equation*}
whose minimizers are solutions to a Poisson problem with a boundary condition of the type $h \frac{\partial u}{\partial \nu} + u = 0$ on $\partial\Omega$ \cite{ab,bu88}. Under these conditions, determining the optimal profile $h$ among all configurations with a fixed mass $m = \int_{\partial\Omega} h \, d\mathcal{H}^{N-1}$ yields an explicit proportionality between the optimal thickness and the boundary temperature.

However, when heat is transferred to the outside through convection, which is a major mode of heat transfer, Robin-type boundary conditions are best suited to the problem. Consider, for example, the heat exchange that occurs on the surfaces of objects like a boiler, a cup of tea, or a building. In these scenarios, the out-flux of heat is proportional to the temperature jump across the body surface. This convective heat exchange has been extensively studied in recent years (see for example \cite{AC25, ACNT24, ant,akk, BBK25, DpNST,DO,DOS,DpNT22}; see also \cite{cnt} for a survey on this topic). Specifically, the functional corresponding to the Robin problem reads as 
\begin{equation*}
	F_{\varepsilon}(v,h) = \frac{1}{2}\int_{\Omega}|\nabla v|^{2}dx + \frac{\varepsilon}{2}\int_{\Sigma_{\varepsilon}}|\nabla v|^{2}dx + \frac{\beta}{2}\int_{\partial\Omega_{\varepsilon}}v^{2}d\mathcal{H}^{N-1} - \int_{\Omega}fvdx,
	\label{eq:F_eps}
\end{equation*}
where $\beta > 0$ is a fixed heat transfer coefficient. The asymptotic analysis as $\varepsilon \to 0$ (see \cite{DpNST}) shows that $F_{\varepsilon,\beta}$ $\Gamma$-converges with respect to the $L^2(\mathbb{R}^N)$ topology to
\begin{equation*}
	F(v,h) = \frac{1}{2}\int_{\Omega}|\nabla v|^{2}dx + \frac{\beta}{2}\int_{\partial\Omega}\frac{v^{2}}{1+\beta h}d\mathcal{H}^{N-1} - \int_{\Omega}fvdx.
	\label{eq:F_limit}
\end{equation*}
In contrast to the simpler case of homogeneous Dirichlet conditions, determining the optimal insulating profile $h$ for the Robin case is a more challenging problem. The optimization requires finding a threshold $c_u$ such that the optimal layer $h_{opt}$ acts only where the boundary temperature exceeds $c_u$, meaning it vanishes completely on cooler parts of the boundary. In particular, if one looks for the configuration that maximizes the heat content of the domain, in \cite{DpNST} it has been proven that the ball surrounded by a uniform distribution of insulating material is the optimal shape.
Moreover, when the domain is not a ball, concentration breaking phenomena are known to appear as noted in \cite{HLL24}.

Further developments extended this framework to eigenvalue problems, investigating operators associated with convective heat transfer. The behavior of the temperature $u(t,x)$ for the corresponding heat equation depends heavily on the first eigenvalue of the elliptic differential operator. Strikingly, such optimization problems reveal a symmetry breaking phenomenon: if the domain is a ball, the convection heat transfer coefficient is sufficiently large, and the total amount of insulation is small enough, the optimal insulating profile and the associated first eigenfunction fail to be radial \cite{BBN17,BBN_Notices,DO}.

\medskip

As a natural extension of these previous investigations, this paper proposes a generalization to a nonlinear setting governed by the $p$-Laplace operator. Nonlinear diffusion operators naturally arise in the modeling of non-Newtonian fluids, glaciology, and nonlinear heat transfer phenomena where the conductivity depends on the temperature gradient. The goal of this paper is to investigate the thermal insulation of a bounded body $\Omega \subset \mathbb{R}^N$ with a prescribed heat source $f$, exchanging heat with the environment through a nonlinear convective mechanism modeled by Robin boundary conditions, with parameter $\beta>0$. 
In our setting, the body is surrounded by an insulating layer $\Sigma_\varepsilon$. However, due to the nonlinear nature of the $p$-Laplacian ($p>1$), the physical scaling of the layer requires careful adjustment. Specifically, the thickness of the layer is assumed to scale as $\varepsilon^{\frac{1}{p-1}} h(\sigma)$, ensuring that the $\Gamma$-limit process correctly captures the effective macroscopic thermal resistance of the boundary.

Furthermore, we extend the classical framework by introducing a non-homogeneous external temperature profile $g$ on the boundary. Considering a variable $g$ allows for the modeling of more realistic thermodynamic scenarios, such as bodies exposed to unevenly heated environments. The optimal insulating layer must now adapt not only to the domain's geometry and the internal heat source, but also to the spatial fluctuations of the external temperature, directly influencing the threshold conditions for the optimal material distribution.

 We finally recall that in this nonlinear framework, very few results are known: see for example \cite{ds} ($g=0$ and $\beta=+\infty$) or \cite{b}, for a different nonlinear model.

\paragraph{Setting of the problem and structure of the paper}

Let $\Omega \subset \mathbb{R}^N (N\ge 2)$ be an open, bounded set with a sufficiently smooth boundary. We model the insulating layer surrounding the domain $\Omega$ as a thin shell constructed along the outward normal direction. Let $p > 1, 0 < \eps < 1$ and let us define
\[
\Sigma_\eps := \left\{\sigma + \eps^{\frac{1}{p-1}} t h(\sigma) \nu(\sigma) : \sigma \in \de \Omega, \; t \in (0, 1)\right\},
\]
where $\nu(\sigma)$ is the outer normal to $\de \Omega$ at $\sigma$, and $h : \de \Omega \to [0, \infty)$ is a nonnegative Lipschitz function that represents the distribution of the layer's thickness at $\sigma \in \de \Omega$. The union of the thermally conductive body and the insulating material is denoted by $\Omega_\eps := \overline{\Omega} \cup \Sigma_\eps$. 

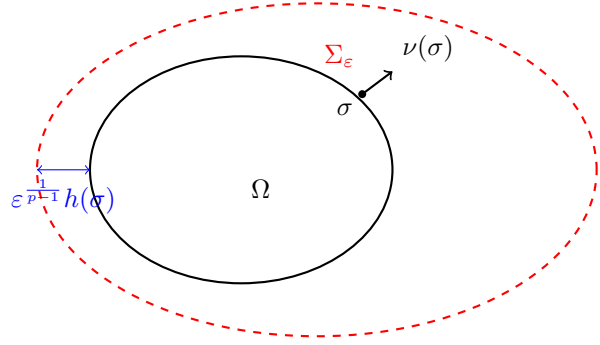
\medskip
\begin{center}
\begin{tikzpicture}
	\draw[thick] (0,0) ellipse (2cm and 1.5cm) node[below right] {\(\Omega\)};
	
	\draw[red, dashed, thick]
	(1,0) ellipse (3.7cm and 2.2cm);  
	
	\draw[thick, ->] (1.6, 1) -- (2, 1.3) node[above right] {\(\nu(\sigma)\)};
	
	\fill (1.6, 1) circle (1.5pt) node[below left] {\(\sigma\)};
	
	\draw[<->, blue, thin] (-2.7,0) -- (-2,0) node[midway, below] {\(\eps^{\frac{1}{p-1}} h(\sigma)\)};
	
	\node[red] at (1.3, 1.5) {\(\Sigma_\eps\)};
\end{tikzpicture}
\end{center}
\medskip

We consider an external temperature profile function $g$, a heat source $f$, and a fixed convective heat transfer parameter $\beta > 0$. The optimization of thermal insulation is governed by the following energy functional
\begin{equation}
\begin{aligned}
	J_\eps(v,h) = \frac{1}{p} \int_\Omega |\nabla v|^p \, dx + \frac{\eps}{p} \int_{\Sigma_\eps} |\nabla v|^p \, dx
	+ \frac{\beta}{p} \int_{\de \Omega_\eps} |v-g|^p \, d\mathcal{H}^{N-1}
	- \int_\Omega fv \, dx, 
	\\ v \in W^{1,p}(\Omega_\eps).
	\label{defJeps}
\end{aligned}	
\end{equation}
Let us stress that, in the entire work, we assume that the function $g$, originally defined on $\partial\Omega$, is extended to the layer $\Sigma_\varepsilon$ constantly along the normal direction. With a slight abuse of notation, we still denote this extension by $g$, namely  $g(\sigma + \eps^{\frac{1}{p-1}}th(\sigma)\nu(\sigma)):= g(\sigma)$ for $\sigma \in \partial\Omega$ and for $t \in [0,1]$.

Let us stress that standard methods in the calculus of variations ensure that a minimizer $u_\eps$ of $J_\eps(v, h)$ exists
\[
J_\eps(u_\eps, h) = \min_{v \in W^{1,p}(\Omega_\eps)} J_\eps(v, h).
\]
This minimizer $u_\eps$ satisfies the following system of equations
\begin{equation}
	\begin{cases} \displaystyle 
		-\Delta_p u_\eps = f & \text{in } \Omega, \\
		\displaystyle -\Delta_p u_\eps = 0 & \text{in } \Sigma_\eps, \\
		\displaystyle |\nabla u_\eps|^{p-2} \frac{\de u_\eps}{\de \nu} + \beta |u_\eps-g|^{p-2} (u_\eps-g) = 0 & \text{on } \de \Omega_\eps, \\
		\displaystyle |\nabla u_\eps^-|^{p-2} \frac{\de u_\eps^-}{\de \nu} = \eps |\nabla u_\eps^+|^{p-2} \frac{\de u_\eps^+}{\de \nu} & \text{on } \de \Omega,
	\end{cases}
	\label{eq:2}
\end{equation}
where $\Delta_p u =  \textrm{div} (|\nabla u|^{p-2} \nabla u)$ is the standard $p$-Laplace operator. Moreover, $\frac{\de u_\eps^-}{\de \nu}$ and $\frac{\de u_\eps^+}{\de \nu}$ represent the normal derivatives of \(u_\eps\) from inside and outside \(\Omega\), respectively. The last equation in \eqref{eq:2} represents the heat transmission condition between the body $\Omega$ and the thin layer $\Sigma_\eps$, reflecting that the layer's thermal conductivity is of order $\varepsilon$.

To study the asymptotic behavior of the sequence of minimizers, we extend the domain of the functional $J_\eps$ to $L^p(\mathbb{R}^N)$ by defining
\begin{equation}
	\label{defJeps2}
J_\eps(v,h) = \infty \quad \text{if } v \in L^p(\mathbb{R}^N) \setminus W^{1,p}(\Omega_\eps).
\end{equation}

The first objective of this paper is to analyze the effective behavior of the thermal insulation as the thickness parameter $\varepsilon$ approaches $0$. In Theorem \ref{teo_gammaconv} below, we prove that, under suitable assumptions on the data, the family of functionals \(J_\eps\) \(\Gamma\)-converges with respect to the strong $L^p$ topology (see Definition \ref{def_gamma} below) to a limit energy $J(v,h)$ defined as 
\begin{equation}
	J(v,h) = \frac{1}{p} \int_\Omega |\nabla v|^p \, dx + \frac{\beta}{p} \int_{\de \Omega} \frac{|v-g|^p}{(1 + \beta^{\frac{1}{p-1}} h)^{p-1}} \, d\mathcal{H}^{N-1} - \int_\Omega fv \, dx, \quad v\in W^{1,p}(\Omega).
	\label{defJ}
\end{equation}
In this limit problem, the geometric layer $\Sigma_\varepsilon$ has vanished, but its insulating effect has been encoded into a modified boundary condition. Consequently, we analyze the associated minimization problem:
\[
J(u,h) = \min_{v\in W^{1,p}(\Omega)} J(v,h).
\]
The minimizer $u \in W^{1,p}(\Omega)$ represents the temperature distribution and solves the following boundary value problem featuring a non-uniform Robin boundary condition
\begin{equation}\label{eqpbmin}
	\begin{cases}
		-\Delta_p u = f & \text{in } \Omega, \\
		\displaystyle |\nabla u|^{p-2} \frac{\partial u}{\partial \nu} + \frac{\beta |u-g|^{p-2}(u-g)}{\left(1+\beta^{\frac{1}{p-1}}h\right)^{p-1}} = 0 & \text{on } \partial\Omega.
	\end{cases}	
\end{equation}

Once the $\Gamma$-limit has been established for a fixed insulation profile $h$, we shift our focus to the problem of finding the optimal distribution of the insulating material. Given a fixed mass $m > 0$, we denote the class of admissible insulating profiles by
\[
\mathcal{H}_m(\partial\Omega) = \left\{ h \in L^1(\partial\Omega) : h \ge 0, \ \int_{\partial\Omega} h d \mathcal{H}^{N-1} = m \right\}.
\]
The goal is to determine the optimal pair $(u, h_u) \in W^{1,p}(\Omega) \times \mathcal{H}_m(\partial\Omega)$ that minimizes the functional $J(v, h)$ jointly in both variables.  In Theorem \ref{teo_min} we show  the existence and uniqueness (provided $\Omega$ is connected) of the optimal configuration. In this context, the optimal thickness $h_u(\sigma)$ is found to be strictly dependent on the boundary values of the temperature $u$ and a specific threshold parameter $c_u$.  Furthermore Theorem \ref{teo_heatcontent} gives an upper bound for the total heat content.

The final section of this work is dedicated to investigating some properties of the optimal insulation $h_u$. In problems governed by standard Dirichlet conditions, classical results assert that the insulating layer surrounds the entire boundary continuously. However, under convective (Robin) conditions, this is no longer guaranteed. 

In Theorem \ref{teo_breaking}, we prove a concentration breaking phenomenon showing that, under suitable conditions on $\Omega$ and on $g$, there exists a critical mass threshold $\overline{m} > 0$. If the total amount of available insulation is small ($m < \overline{m}$), the optimal material does not distribute itself uniformly, and it vanishes ($h=0$) on a subset of the boundary $\partial\Omega$ with positive measure.
In particular, the existence of such threshold value happens when $\Omega$ is not a ball and $g$ is constant, or, in the case of the ball, when $g$ is not constant (Corollaries \ref{cor_nonpalla} and \ref{cor_palla_noncostante}). This behavior stands in contrast to the radially symmetric case of the ball and $g$ is constant, where the insulation remains strictly positive and uniform everywhere, regardless of the mass $m$. Moreover, unexpected phenomena may appear depending on the topology of $\de\Omega$. For example, if $\Omega$ is the annulus in the plane, under suitable conditions on the difference of the temperature inside and outside $\Omega$, there exist  disjoint intervals of values of $m$ in such a way concentration breaking appears (see Example \ref{exane}), also showing that the assumptions of Theorem \ref{teo_breaking} are optimal.

\medskip

The paper is organized as follows. In Section 2, we prove the $\Gamma$-convergence of the functionals $J_\varepsilon$ to $J$. In Section 3, we analyze the limit functional, showing the existence and uniqueness of the optimal pair $(u, h_u)$ and providing the exact characterization of the insulating profile $h_u$. Furthermore, we establish an optimal upper bound for the total heat content. Finally, in Section 4, we deal with the concentration breaking phenomenon, proving the existence of the critical mass threshold $\overline{m}$ and analysing the behavior of the optimal layer. Moreover, we conclude the section with examples and remarks showing the optimality of our hypotheses. 

 \section{$\Gamma$-Convergence}
  
 This section is dedicated to the study of the $\Gamma$-convergence of the functional $J_\eps$ towards $J$ (which are defined in $\eqref{defJeps},\eqref{defJeps2},\eqref{defJ}$) as $\varepsilon\to 0^+$. Let us also stress that, for the entire section, $h : \de \Omega \to (0, \infty)$ is a fixed Lipschitz function. In order to be self-contained, we recall the notion of $\Gamma$-convergence for $J_\eps$.
 
 \begin{definition}\label{def_gamma}
 	The functional $J_\eps$ $\Gamma$-converges in $L^p(\mathbb{R}^N)$ to $J$ as $\eps\to 0^+$ if  for any $v\in L^p(\mathbb{R}^N)$, it holds:
 	\begin{itemize}  	
 		\item[i)] \textit{liminf inequality}:  for any sequence $v_\eps \in L^p(\mathbb{R}^N)$ which  converges to $v$ in $L^p(\mathbb{R}^N)$ as $\eps\to 0^+$, it holds
 		\begin{equation}\label{liminf}
 			\liminf_{\eps \to 0^+} J_\eps(v_\eps,h) \ge J(v,h);
 		\end{equation}
 		\item[ii)] \textit{limsup inequality}: there exists a sequence $v_\eps \in L^p(\mathbb{R}^N)$ which converges to $v$ in $L^p(\mathbb{R}^N)$ as $\eps\to 0^+$ such that
 		\begin{equation}\label{limsup}
 			\limsup_{\eps\to 0^+} J_\eps(v_\eps,h) \le J(v,h).
 		\end{equation}
 	\end{itemize}
 \end{definition}
 We are now ready to state and prove the main result of this section; the following $\Gamma$-convergence result holds:
 \begin{theorem}\label{teo_gammaconv}
 		Let $p>1$, let $g,h$ be Lipschitz functions defined on $\partial\Omega$ with $h> 0$ on $\partial\Omega$, let $\beta > 0$ and let $f\in L^{\frac{p}{p-1}}(\Omega)$. Then the functional $J_\eps$ (defined in \eqref{defJeps}, \eqref{defJeps2}) $\Gamma$-converges to the functional $J$ (defined in \eqref{defJ}) in the sense of Definition \ref{def_gamma}.  
 \end{theorem}
\begin{proof}
 We first prove the liminf inequality.  Then let $v \in L^p(\mathbb{R}^N)$ and let $v_\eps\in L^p(\mathbb{R}^N)$ converge to $v$ strongly in $L^p(\mathbb{R}^N)$ as $\eps \to 0^+$. We aim to show the validity of \eqref{liminf}.
	Without losing  generality we can then assume that $v_\eps\in W^{1,p}(\Omega_\eps)$, that $\liminf_{\eps \to 0} J_\eps(v_\eps) < +\infty$ and that the liminf is actually a limit; therefore a simple calculation implies that $v_\eps$  is bounded in $W^{1,p}(\Omega)$ with respect to $\eps$ meaning that $v_\eps \rightharpoonup v$ weakly in $W^{1,p}(\Omega)$.
 	
 	By weak lower semicontinuity and the strong convergence of $v_\eps$ in $L^{p}(\Omega)$ one gains
 	\begin{equation}\label{liminf1}
 	\liminf_{\eps \to 0^+} \left( \frac{1}{p}\int_{\Omega}\left|\nabla v_\eps\right|^{p}\,dx - \int_{\Omega}fv_\eps \,dx \right) \ge \frac{1}{p}\int_{\Omega}\left|\nabla v\right|^{p}\,dx - \int_{\Omega}fv \,dx.
 	\end{equation}
 	Then it remains to show the lower bound for the layer and for the outer boundary terms, namely we need to prove that
 	\begin{equation} \label{liminf1bis}
 		\liminf_{\eps \to 0^+} \left( \frac{\eps}{p}\int_{\Sigma_\eps}\left|\nabla v_{\eps}\right|^{p}\,dx + \frac{\beta}{p}\int_{\de\Omega_\eps}\left|v_{\eps}-g\right|^{p}\,d\mathcal{H}^{N-1} \right) \ge \frac{\beta}{p}\int_{\de\Omega} \frac{\left|v-g\right|^{p}}{(1 + \beta^{\frac{1}{p-1}} h)^{p-1}} \,d\mathcal{H}^{N-1}.
 	\end{equation}
 	For the first term on the left-hand of \eqref{liminf1bis} we use the change of variables $x = \sigma + s\nu(\sigma)$ for $x \in \Sigma_\eps$, $s \in (0, \eps^{\frac{1}{p-1}} h(\sigma))$ so that the volume element is given by $dx = \mathcal J(\sigma, s) \,d\mathcal{H}^{N-1}(\sigma) ds$, where $\mathcal J(\sigma, s) = 1 + o(1)$ uniformly on $\Sigma_\eps$ for $\varepsilon\to 0^+$ as both $\nu$ and $h$ are Lipschitz continuous functions. Thus there exists $C>0$ such that $\mathcal J(\sigma, s) \ge 1 - C\eps^{\frac{1}{p-1}}$ for $\eps$ small enough. Hence one has
 	\begin{equation}\label{liminf2}
 	\begin{aligned}
 		&\frac{\eps}{p}\int_{\Sigma_\eps}\left|\nabla v_{\eps}\right|^{p}\,dx = \frac{\eps}{p} \int_{\de\Omega} \int_0^{\eps^{\frac{1}{p-1}} h(\sigma)} \left|\nabla v_\eps(\sigma+s\nu(\sigma))\right|^p \mathcal J(\sigma, s) \,ds \,d\mathcal{H}^{N-1}(\sigma) \\
 		&\ge \frac{\eps(1-C \eps^{\frac{1}{p-1}})}{p} \int_{\de\Omega} \left( \int_0^{\eps^{\frac{1}{p-1}} h(\sigma)} \left|\nabla v_\eps(\sigma+s\nu(\sigma))\right|^p \,ds \right) \,d\mathcal{H}^{N-1}(\sigma).
 	\end{aligned}
 	\end{equation}
 	Now observe that it follows from $|\nabla v_\eps|^p \ge |\nabla v_\eps \cdot \nu|^p$ and from the H\"older inequality that
 	\[
 	\begin{aligned}
 		\int_0^{\eps^{\frac{1}{p-1}} h(\sigma)} |\nabla v_\eps(\sigma+s\nu(\sigma))|^p \,ds &\ge \int_0^{\eps^{\frac{1}{p-1}} h(\sigma)} |\nabla v_\eps(\sigma+s\nu(\sigma)) \cdot \nu(\sigma)|^p \,ds \\
 		&\ge \frac{1}{(\eps^{\frac{1}{p-1}} h(\sigma))^{p-1}} \left( \int_0^{\eps^{\frac{1}{p-1}} h(\sigma)} |\nabla v_\eps(\sigma+s\nu(\sigma)) \cdot \nu(\sigma)| \,ds \right)^p \\
 		&\ge \frac{1}{(\eps^{\frac{1}{p-1}} h(\sigma))^{p-1}} \left| \int_0^{\eps^{\frac{1}{p-1}} h(\sigma)} \nabla v_\eps(\sigma+s\nu(\sigma)) \cdot \nu(\sigma) \,ds \right|^p \\
 		&= \frac{\left|v_\eps(\sigma + \eps^{\frac{1}{p-1}} h(\sigma)\nu(\sigma)) - v_\eps(\sigma )\right|^p}{(\eps^{\frac{1}{p-1}} h(\sigma))^{p-1}}.
 	\end{aligned}
 	\]
 	Then one can gather the previous in \eqref{liminf2}, yielding to
	\begin{equation}\label{liminf3}
	\begin{aligned}
	\frac{\eps}{p}\int_{\Sigma_\eps}\left|\nabla v_{\eps}\right|^{p}\,dx &\ge \frac{1-C \eps^{\frac{1}{p-1}}}{p} \int_{\de\Omega} \frac{\left|v_\eps(\sigma + \eps^{\frac{1}{p-1}} h(\sigma)\nu(\sigma)) - v_\eps(\sigma )\right|^p}{h(\sigma)^{p-1}} \,d\mathcal{H}^{N-1}.
\end{aligned}
	\end{equation}
	
	For the second term on the left-hand of \eqref{liminf1bis} one can set $x=\sigma + \eps^{\frac{1}{p-1}} h(\sigma)\nu(\sigma)$ and, as $h$ is a Lipschitz function, one can reason on the Jacobian similar to what was done for the term on $\Sigma_\eps$; this leads to
	\begin{equation}\label{liminf4}
	\begin{aligned}
 	&\frac{\beta}{p}\int_{\de\Omega_\eps}\left|v_{\eps}-g\right|^{p}\,d\mathcal{H}^{N-1} 
 	\\
 	&\ge  \frac{\beta(1-C\eps^{\frac{1}{p-1}})}{p} \int_{\de\Omega} \left|v_\eps(\sigma + \eps^{\frac{1}{p-1}} h(\sigma)\nu(\sigma))- g(\sigma)\right|^p \,d\mathcal{H}^{N-1},
 	\end{aligned}
 	\end{equation}
 	for some constant $C>0$ and for $\eps$ small enough as, recall, we set $g(\sigma + \eps^{\frac{1}{p-1}}h(\sigma)\nu(\sigma)):= g(\sigma)$. 
 	Then, having in force \eqref{liminf3} and \eqref{liminf4} and in order to prove \eqref{liminf1bis}, we need
 	\begin{equation*}
 	\begin{aligned}
 		&\frac{1-C \eps^{\frac{1}{p-1}}}{p} \int_{\de\Omega} \left(\frac{\left|v_\eps(\sigma + \eps^{\frac{1}{p-1}} h(\sigma)\nu(\sigma)) - v_\eps(\sigma )\right|^p}{h(\sigma)^{p-1}}  +\beta \left|v_\eps(\sigma + \eps^{\frac{1}{p-1}} h(\sigma)\nu(\sigma))-g(\sigma)\right|^p \right) \,d\mathcal{H}^{N-1} 
 		\\
 		& \ge 
 		\frac{\beta}{p}\int_{\de\Omega} \frac{\left|v-g\right|^{p}}{(1 + \beta^{\frac{1}{p-1}} h)^{p-1}} \,d\mathcal{H}^{N-1},
 	\end{aligned}
 	\end{equation*}
 	which means proving 
 		{\small \begin{equation}\label{liminf4bis}
 		\begin{aligned}
 			&\liminf_{\eps \to 0^+}\frac{1}{p} \int_{\de\Omega} \left(\frac{\left|v_\eps(\sigma + \eps^{\frac{1}{p-1}} h(\sigma)\nu(\sigma)) - v_\eps(\sigma )\right|^p}{h(\sigma)^{p-1}}  +\beta \left|v_\eps(\sigma + \eps^{\frac{1}{p-1}} h(\sigma)\nu(\sigma))- g(\sigma)\right|^p \right) \,d\mathcal{H}^{N-1}
 			\\
 			& \ge 
 			\frac{\beta}{p}\int_{\de\Omega} \frac{\left|v-g\right|^{p}}{(1 + \beta^\frac{1}{p-1} h)^{p-1}} \,d\mathcal{H}^{N-1},
 		\end{aligned}
 	\end{equation}
 }
 	as it holds that 
 	\begin{equation*}
 	\begin{aligned}
 		&\frac{C \eps^{\frac{1}{p-1}}}{p} \int_{\de\Omega} \left(\frac{\left|v_\eps(\sigma + \eps^{\frac{1}{p-1}} h(\sigma)\nu(\sigma)) - v_\eps(\sigma )\right|^p}{h(\sigma)^{p-1}}  +\beta \left|v_\eps(\sigma + \eps^{\frac{1}{p-1}} h(\sigma)\nu(\sigma))-g(\sigma)\right|^p \right) \,d\mathcal{H}^{N-1} 
 	\end{aligned}
 \end{equation*}
converges to $0$ as $\eps \to 0^+$. Indeed, thanks to \eqref{liminf3} and \eqref{liminf4}, the integral is bounded by a constant multiple of $J_\eps(v_\eps)$, which is uniformly bounded for $\eps>0$.

 Now observe that, applying \eqref{lem_dis} with $a= v_\eps(\sigma + \eps^{\frac{1}{p-1}} h(\sigma)\nu(\sigma))- g(\sigma) $ and $b= v_\eps(\sigma) - g(\sigma)$ and with $\lambda= h$ and $\eta= \beta$, one has
 	\begin{equation}\label{liminf5}
 \begin{aligned}
 	&\frac{1}{p}\int_{\de\Omega} \left( \frac{\left|v_\eps(\sigma + \eps^{\frac{1}{p-1}} h(\sigma)\nu(\sigma)) - v_\eps(\sigma )\right|^p}{h(\sigma)^{p-1}} + \beta \left|v_{\eps}(\sigma+\eps^{\frac{1}{p-1}} h(\sigma)\nu(\sigma)) - g(\sigma)\right|^{p} \right) d\mathcal{H}^{N-1} 
 	\\
 	&\ge \frac{\beta}{p}\int_{\de\Omega}  \frac{\left| v_\eps(\sigma )- g(\sigma)\right|^p}{(1+ \beta^{\frac{1}{p-1}}h(\sigma))^{p-1}}   d\mathcal{H}^{N-1}.
 \end{aligned}
 \end{equation}
Now an application of Fatou's Lemma as $\eps\to 0^+$ into inequality \eqref{liminf5} gives the validity of \eqref{liminf4bis}.
In particular this shows \eqref{liminf1bis} which, jointly with \eqref{liminf1}, proves the liminf inequality \eqref{liminf}.

\medskip

Let us now focus on proving the limsup inequality. Without loss of generality, we consider the nontrivial case 
$v\in L^p(\mathbb{R}^N)$ such that
$v|_\Omega\in W^{1,p}(\Omega)$, and consider
\[
\widehat v:=E(v|_\Omega)\in W^{1,p}(\mathbb{R}^N)
\]
a Sobolev extension of its restriction to $\Omega$.
We define $v_\eps\in L^p(\mathbb{R}^N)$ as
\[
v_\eps(x):=
\begin{cases}
\displaystyle
\widehat v(x)
-\frac{\beta^{\frac1{p-1}}d(x)
(\widehat v(x)-g(x))}
{\eps^{\frac1{p-1}}
(1+\beta^{\frac1{p-1}}h(x))},
&x\in\Omega_\eps,\\[8pt]
v(x),&x\in\mathbb{R}^N\setminus\Omega_\eps,
\end{cases}
\]
where $d(x)$ represents the distance from $x$ to $\Omega$.  In particular, it is worth mentioning that also $h$, as well as we did with $g$, is extended into the layer $\Sigma_{\varepsilon}$ constantly along the outward normal direction. Therefore, since $v=\widehat v$ almost everywhere in $\Omega$, the functions
$v_\eps$ and $v$ differ only in $\Sigma_\eps$. Moreover,
\[
\|v_\eps-v\|_{L^p(\mathbb{R}^N)}^p
\leq C\int_{\Sigma_\eps}
\bigl(|\widehat v|^p+|v|^p+|g|^p\bigr)\,dx
\stackrel{\eps \to 0}{\to } 0.
\]
Thus $v_\eps\to v$ strongly in $L^p(\mathbb{R}^N)$. Then, as $v_{\varepsilon}=v$ in $\Omega$, we just need to show that
\begin{equation}\label{limsup1}
\limsup_{\varepsilon \rightarrow 0}\left(\frac{\varepsilon}{p} \int_{\Sigma_{\varepsilon}}\left|\nabla v_{\varepsilon}\right|^p d x+\frac{\beta}{p} \int_{\partial \Omega_{\varepsilon}}\left|v_{\varepsilon}-g\right|^p d \mathcal{H}^{N-1}\right) \leqslant \frac{\beta}{p} \int_{\partial \Omega} \frac{|v-g|^p}{\left(1+\beta^{\frac{1}{p-1}} h\right)^{p-1}} d \mathcal{H}^{N-1}.
\end{equation}
We start analyzing the layer term; through the same change of variable already employed for proving the liminf inequality, we get
\begin{equation}\label{limsup2bis}
\begin{aligned}
	\frac{\varepsilon}{p} \int_{\Sigma_{\varepsilon}}\left|\nabla v_{\varepsilon}\right|^p d x & =\frac{\varepsilon}{p} \int_{\partial \Omega}\left(\int_0^{\varepsilon^{\frac{1}{p-1}} h(\sigma)}\left|\nabla v_{\varepsilon}(\sigma+s \nu(\sigma))\right|^p J(\sigma, s) d s\right) d \mathcal{H}^{N-1}(\sigma) \\
	& \leqslant \frac{\varepsilon\left(1+C \varepsilon^{\frac{1}{p-1}}\right)}{p} \int_{\partial \Omega}\left(\int_0^{\varepsilon^{\frac{1}{p-1}} h(\sigma)}\left|\nabla v_{\varepsilon}(\sigma+s \nu(\sigma))\right|^p d s\right) d \mathcal{H}^{N-1}(\sigma)
\end{aligned}
\end{equation}
since $\mathcal{J}(\sigma, s) \leqslant 1+C \varepsilon^{\frac{1}{p-1}}$ for $\varepsilon$ small enough. In particular one can pass to the limit $\varepsilon \rightarrow 0$ in \eqref{limsup2bis} in order to obtain
\begin{equation}\label{limsup2}
\limsup_{\varepsilon \rightarrow 0} \frac{\varepsilon}{p} \int_{\Sigma_{\varepsilon}}\left|\nabla v_{\varepsilon}\right|^p d x \leqslant \frac{\beta}{p} \int_{\partial \Omega} \frac{\beta^{\frac{1}{p-1}} h|v-g|^p}{\left(1+\beta^{\frac{1}{p-1}} h\right)^p} d \mathcal{H}^{N-1},
\end{equation}
where we used that $\left|\nabla v_{\varepsilon}\right|^p \leqslant\left(1+\varepsilon^{\frac{1}{p-1}}\right)^{p-1}\left( \frac{\beta^{\frac{p}{p-1}}| \widehat v-g|^p}{\varepsilon^{\frac{p}{p-1}}\left(1+\beta^{\frac{1}{p-1}} h\right)^p}+\frac{j(x)}{\varepsilon}\right)$ where $j$ is a summable function not depending on $\varepsilon$. Indeed the previous estimate can be checked by calculating the normal and tangential components of the gradient and using also (A.2) by fixing $\gamma^{-1}=\left(1+\varepsilon^{\frac{1}{p-1}}\right)$ and noting that the tangential gradient is actually bounded in $\varepsilon$.

Now, we focus on the second term on the left-hand of \eqref{limsup1}; first observe that a change of variable yields
	\begin{equation}
\begin{aligned}\label{limsup3_0}
	&\int_{\partial \Omega_{\varepsilon}} |v_{\eps}-g|^p d \mathcal{H}^{N-1}
	\\
	&=\int_{\partial \Omega} |v_{\eps}(\sigma+\eps^{\frac{1}{p-1}} h(\sigma) \nu(\sigma))-g(\sigma+\eps^{\frac{1}{p-1}} h(\sigma) \nu(\sigma))|^p \mathcal{J}(\sigma) d \mathcal{H}^{N-1},
\end{aligned}
\end{equation}
where one has $\mathcal{J}(\sigma) \le 1 + C\eps^{\frac{1}{p-1}}$.
Now observe that, on $\partial\Omega_\eps$ where clearly $d= \eps^{\frac{1}{p-1}}h$, one has 
\begin{equation}\label{limsup3}
v_{\varepsilon}-g = (\widehat v - g) -  \frac{\beta^{\frac{1}{p-1}}h(\widehat v -g)}{1+\beta^{\frac{1}{p-1}}h} = (\widehat v - g) \left( 1 - \frac{\beta^{\frac{1}{p-1}}h}{1+\beta^{\frac{1}{p-1}}h} \right) =  \frac{\widehat v - g}{1+\beta^{\frac{1}{p-1}}h}.
\end{equation}
Therefore, it follows from \eqref{limsup3_0} and \eqref{limsup3} that
\begin{equation}\label{limsup4}
	 \lim_{\eps\to 0}\frac{\beta}{p} \int_{\de \Omega_{\varepsilon}}\left|v_{\varepsilon}-g\right|^p d \mathcal{H}^{N-1} = \frac{\beta}{p}\int_{\partial\Omega} \frac{|v-g|^p}{(1+\beta^{\frac{1}{p-1}}h)^p} \, d\mathcal{H}^{N-1},
\end{equation}
where we used that $g(\sigma)$ is constant along the normal direction to $\partial\Omega$. 
Now gathering together \eqref{limsup2} and \eqref{limsup4} one obtains \eqref{limsup1}. 
This concludes the proof.
\end{proof}

 \section{The study of the optimal insulation}
 
 In this section we deal with a varying $h$ in order to achieve the optimal insulation for $\Omega$. To this aim we fix the total amount of insulating material that is, for $m>0$, we set
 \[
 \mathcal{H}_m(\partial\Omega) = \left\{ h \in L^1(\partial\Omega) : h \ge 0, \ \int_{\partial\Omega} h d \mathcal{H}^{N-1} = m \right\}.
 \]
 Here we analyse the limit functional $J(v,h)$ defined in \eqref{defJ} by studying the following minimization problem
 \begin{equation}\label{pbmin}
 \min_{(v,h)\in W^{1,p}(\Omega) \times \mathcal{H}_m(\partial\Omega)}J(v,h).
 \end{equation} 
 The main theorem of this section is the following.

 \begin{theorem}\label{teo_min}
 	Let $p>1$, let $\beta, m>0$, let $g \in L^p(\partial\Omega)$ and let 
 	$f\in L^{\frac{p}{p-1}}(\Omega)$.
 	Then there exists a pair $(u,h_u) \in W^{1,p}(\Omega) \times \mathcal{H}_m(\partial\Omega)$ minimizing \eqref{pbmin}.
 	Moreover, the optimal distribution $h_u$ is given by
 	\begin{equation}\label{defh}
 		h_u(\sigma) := \begin{cases}
 			\displaystyle \frac{|u(\sigma)-g(\sigma)|}{c_{u}\beta^{\frac{1}{p-1}}} - \frac{1}{\beta^{\frac{1}{p-1}}} & \text{if } |u(\sigma)-g(\sigma)| > c_u, \\
 			0 & \text{otherwise,}
 		\end{cases}
 	\end{equation}
where $c_u \ge 0$ is the unique constant satisfying
 \begin{equation}\label{eq:c_definition}
 	c_u = \left( \frac{1}{|\{ |u-g| > c_u \}| + m\beta^{\frac{1}{p-1}}} \right) \int_{\{ |u-g| > c_u \}} |u-g| \, d\mathcal{H}^{N-1}.
 \end{equation}
 Moreover $c_u = 0$ if and only if $u = g \ \mathcal{H}^{N-1}\text{-a.e. on } \partial\Omega$.	
 
 The pair $(u,h_u)$ is also a solution to \eqref{eqpbmin}. Finally, if the domain $\Omega$ is connected and $\int_\Omega f \, dx \neq 0$, the minimizing pair $(u,h_u)$ is unique.
 \end{theorem}

\begin{example}\label{ex:counterexample}
	If $\int_\Omega f \, dx = 0$, the uniqueness of the optimal insulating profile $h$ can fail. For simplicity, let us consider the case in which $\Omega = B_1$ is the unit ball in $\mathbb{R}^N$, $p=2$, $g=0$, and let $\beta, m > 0$. 
	
	Consider the radial function $u(x) = \frac{(1 - |x|^2)^2}{4} $. It is immediate to check that $u=0$ and $\frac{\partial u}{\partial \nu} = 0$ on $\partial B_1$, and $u$ satisfies
	\[
 -\Delta u(x) = f(x)=N - (N+2)|x|^2,
	\]
	 where $\int_\Omega f \ dx = 0$.
	For any admissible insulation profile $h \in \mathcal{H}_m(\partial B_1)$, the function $u$ trivially satisfies the Robin boundary condition $\frac{\partial u}{\partial \nu} + \frac{\beta u}{1+\beta h} = 0$ on $\partial B_1$, since both terms vanish identically. Moreover, as the functional $J(\cdot, h)$ is strictly convex, $u$ is its unique global minimizer in $W^{1,2}(B_1)$ for any fixed $h$. The corresponding minimum energy is given by
	\[
	J(u, h) = \frac{1}{2} \int_{B_1} |\nabla u|^2 \, dx - \int_{B_1} f u \, dx,
	\]
	which is independent of $h$ due to the fact that $u=0$ on $\partial B_1$. 
	Therefore, the pair $(u, h)$ is a global minimizer of the joint problem \eqref{pbmin} for any choice of $h \in \mathcal{H}_m(\partial B_1)$. This shows that while the optimal temperature $u$ remains unique, the optimal insulation $h_u$ is not unique.
	
	This example highlights the physical meaning of the condition $\int_\Omega f \, dx \neq 0$: it prevents the body from perfectly insulating itself ($\frac{\partial u}{\partial \nu}=0$ on $\partial\Omega$). If $\int_\Omega f \neq 0$, the trace of $u$ cannot vanish everywhere, ensuring $c_u > 0$ and thus selecting a unique optimal configuration $h_u$.

	Furthermore, the requirement that $\Omega$ is connected is also necessary. To see this, assume for simplicity $p=2$, $f=1$ and $g=0$, and let $\Omega$ be composed of two identical disjoint balls, say $\Omega = B_1 \cup B_2$. Let $m>0$ be large enough. Then, there exists a symmetric global minimizer $(u_0, h_0)$, where $m$ is equally split between $B_1$ and $B_2$, and $h_0 > 0$ everywhere on $\partial \Omega$. 
Integrating $-\Delta u_0 = 1$ over $B_1$ and using the Robin boundary condition, we obtain
\[
|B_1| = \int_{\partial B_1} \frac{\beta u_0}{1+\beta h_0}\,d\mathcal{H}^{N-1}.
\]
Since $h_0$ satisfies the optimality condition $1+\beta h_0 = \frac{u_0}{c_{u_0}}$ on the boundary, we deduce
\[
|B_1| = \beta c_{u_0} P(B_1).
\]
Now, for any sufficiently small constant $k \in \mathbb{R}$ (such that $h$ remains non-negative), we can construct a new configuration $(u_k, h_k)$. By exploiting the fact that the domain is not connected, we shift the temperatures and the insulating mass with opposite signs on each component:
\[
u_k = u_0 + k, \quad h_k = h_0 + \frac{k}{c_{u_0} \beta} \quad \text{on } B_1,
\]
\[
u_k = u_0 - k, \quad h_k = h_0 - \frac{k}{c_{u_0} \beta} \quad \text{on } B_2.
\]
As $P(B_1) = P(B_2)$, the total mass remains strictly invariant, that is $\int_{\partial\Omega} h_k \, d\mathcal{H}^{N-1} = m$. Moreover, by construction, the ratio $\frac{u_k}{1+\beta h_k}$ is identically equal to $c_{u_0}$ on the whole boundary $\partial\Omega$, matching perfectly $\frac{u_0}{1+\beta h_0}$. This implies that $c_{u_k} = c_{u_0}$, the local optimality condition is preserved (see also Proposition \ref{prop_convexity}, equation \eqref{cond_conv}).

To verify that $(u_k, h_k)$ is still a global minimizer, we compute the energy variation. The gradient term is unaffected by constant shifts, and the boundary integrand simplifies to $\frac{\beta}{2} \frac{u_k^2}{1+\beta h_k} = \frac{\beta c_{u_0}}{2} u_k$. Thus, the energy variation on $B_1$ is simply
\[
\delta_1 = k \left( \frac{\beta c_{	u_0}}{2} P(B_1) - |B_1| \right) = -\frac{k}{2} |B_1|.
\] 
Analogously, the energy variation on $B_2$ is $\delta_2 = \frac{k}{2} |B_2|$. 
Since $|B_1|=|B_2|$, the total energy variation vanishes. This shows that the problem admits infinitely many global minimizers. \triang
\end{example}

Before proving the main theorem, we establish a few preliminary results, starting with a proposition concerning the boundary term minimization.
\begin{proposition}\label{prop_bordo}
	Let $p>1$, let $\beta,m > 0$, let $v,g \in L^p(\partial\Omega)$, and let $h_v \in L^1(\partial\Omega)$ be the function defined by \eqref{defh} where $c_v$ is the unique constant given by \eqref{eq:c_definition} and it holds $c_v = 0$ if and only if $v = g \ \mathcal{H}^{N-1}\text{-a.e. on } \partial\Omega$.	
	Then $h_v$ is a solution to the minimum problem
\begin{equation}\label{probminh}
	\min_{\hat{h}\in\mathcal{H}_m(\partial\Omega)}\int_{\partial\Omega}\frac{|v-g|^p}{(1+\beta^{\frac{1}{p-1}}\hat{h})^{p-1}}d\mathcal{H}^{N-1}.
	\end{equation}
	In particular $h_v$ is the unique minimizer if $v-g$ is not identically null on $\partial\Omega$.
\end{proposition} 
 \begin{proof}
 	The case with $v-g$ identically null is trivial, as the integrand is zero, and any $\hat{h} \in \mathcal{H}_m(\partial\Omega)$ is a minimizer. Then without loss of generality, we focus on the case where $v-g\not\equiv 0$.
 	
 	\medskip

 	First we show the existence of a unique $c_v>0$ satisfying \eqref{eq:c_definition} reasoning as in \cite[Lemma $4.1$]{DpNST}. 
 	To this aim we define the functions $\xi_1, \xi_2: [0, \infty) \to [0, \infty)$ by:
 	\begin{align*}
 		\xi_1(t) &:= \int_{\{|v-g| > t\}} (|v(\sigma)-g(\sigma)| - t) d\mathcal{H}^{N-1}, \\
 		\xi_2(t) &:= mt\beta^{\frac{1}{p-1}},
 	\end{align*}
 	and we aim to prove that there exists a unique value $c_v$ such that $\xi_1(c_v) = \xi_2(c_v)$ which is equivalent to showing \eqref{eq:c_definition}. 
 	First observe that $\xi_1$ is continuous and non-increasing while $\xi_2$ is continuous and strictly increasing, as $m>0$ and $\beta>0$.
 	
 	Moreover observe that $\xi_2(0)=0$ and $\xi_1(0) = \int_{\partial\Omega} |v-g| d\mathcal{H}^{N-1} > 0$ as $v-g$ is not null and we also have  $\lim_{t\to\infty} \xi_2(t) = \infty$.
 	This is sufficient to deduce a unique intersection point $c_v \in (0, \infty)$ such that $\xi_1(c_v) = \xi_2(c_v)$.
 	  	
  	\medskip
  	
  	Next we show that the nonnegative function $h_v$, defined in \eqref{defh},  belongs to $\mathcal{H}_m(\partial\Omega)$.
 	Indeed one has
 	\begin{equation}\label{hinm}
 	\begin{aligned}
 		\int_{\partial\Omega} h_v d\mathcal{H}^{N-1} &= \int_{\{|v-g| > c_v\}} \left( \frac{|v(\sigma)-g(\sigma)|}{c_v\beta^{\frac{1}{p-1}}} - \frac{1}{\beta^{\frac{1}{p-1}}} \right) d\mathcal{H}^{N-1} \\
 		&= \frac{1}{\beta^{\frac{1}{p-1}}} \left( \frac{1}{c_v} \int_{\{|v-g| > c_v\}} |v-g| d\mathcal{H}^{N-1} - |\{|v-g| > c_v\}|\right).
 	\end{aligned}
 	\end{equation}
 	Now observe that \eqref{eq:c_definition} means $ \int_{\{|v-g| > c_v\}} |v-g| d\mathcal{H}^{N-1} = c_v \left( |\{|v-g| > c_v\}| + m\beta^{\frac{1}{p-1}} \right)$; then from \eqref{hinm}, one gets
 	\[
 	\int_{\partial\Omega} h_v d\mathcal{H}^{N-1} = \frac{1}{\beta^{\frac{1}{p-1}}} \left( \frac{c_v \left( |\{|v-g| > c_v\}| + m\beta^{\frac{1}{p-1}} \right)}{c_v} - |\{|v-g| > c_v\}| \right) = m,
 	\]
 	namely $h_v \in \mathcal{H}_m(\partial\Omega)$.
 	
 	\medskip
 	
 Now we are left to show the optimality of $h_v$ for problem \eqref{probminh}, namely that
 	\[ \int_{\partial\Omega}\frac{|v-g|^p}{(1+\beta^{\frac{1}{p-1}}h_v)^{p-1}}d\mathcal{H}^{N-1}=\min_{\hat{h}\in\mathcal{H}_m(\partial\Omega)}\int_{\partial\Omega}\frac{|v-g|^p}{(1+\beta^{\frac{1}{p-1}}\hat{h})^{p-1}}d\mathcal{H}^{N-1}.
 	\]	
  	  	 	 	 	 
 	Then let $\hat{h} \in \mathcal{H}_m(\partial\Omega)$ and let us set $H_t = \hat{h} + t(h_v - \hat{h})$ for $t \in [0,1]$. In particular let us define the following function
 	\[
 	\psi(t) := \int_{\partial\Omega} \frac{|v-g|^p}{\left(1+\beta^{\frac{1}{p-1}}H_t\right)^{p-1}} d\mathcal{H}^{N-1}.
 	\]
 	To show that $h_v$ is a minimizer, it is sufficient to prove that $\psi'(t) \le 0$ for $t \in (0,1)$ so that the minimum would be attained at $\psi(1)$. 
 	Then, one has
 	\[
 	\psi'(t) = -(p-1)\beta^{\frac{1}{p-1}} \int_{\partial\Omega} \frac{|v-g|^p(h_v-\hat{h})}{\left(1+\beta^{\frac{1}{p-1}}H_t\right)^{p}} d\mathcal{H}^{N-1}.
 	\]
	In order to prove that the previous expression is nonpositive, let us set $A := \{\sigma \in \partial\Omega : h_v(\sigma) > \hat{h}(\sigma)\}$ and $B := \{\sigma \in \partial\Omega : h_v(\sigma) < \hat{h}(\sigma)\}$.

 Now, as on $A$, $h_v > \hat{h}$, this implies that for any $t \in (0,1)$, it holds $H_t < h_v$; then, on $A$, it follows that 
 \begin{equation}\label{ineh}
 -(p-1)\beta^{\frac{1}{p-1}} \frac{|v-g|^p(h_v-\hat{h})}{\left(1+\beta^{\frac{1}{p-1}}
 H_t\right)^{p}} < -(p-1)\beta^{\frac{1}{p-1}} \frac{|v-g|^p(h_v-\hat{h})}{\left(1+\beta^{\frac{1}{p-1}}h_v\right)^{p}}.	
 \end{equation}
 On $B$, it holds $h_v< \hat{h}$ so that $H_t > h_v$; then a similar reasoning to the previous leads to the same inequality for the integrand given by \eqref{ineh} as, in this case, $h_v - \hat{h}$ is nonpositive.
 
 Then, as if $h_v= \hat{h}$ then the integrand is null, in any case we have shown that 
 \[
 \psi'(t) \le -(p-1)\beta^{\frac{1}{p-1}} \int_{\partial\Omega} \frac{|v-g|^p(h_v-\hat{h})}{\left(1+\beta^{\frac{1}{p-1}}h_v\right)^{p}} d\mathcal{H}^{N-1},
 \]
 and now we aim to show that 
 \[
 I := \int_{\partial\Omega} \frac{|v-g|^p(h_v-\hat{h})}{\left(1+\beta^{\frac{1}{p-1}}h_v\right)^{p}} d\mathcal{H}^{N-1} \ge 0.
 \]
Then, recalling \eqref{defh}, one has
 \begin{align*}
 	I &= \int_{\{|v-g| > c_v\}} c_v^p(h_v - \hat{h}) d\mathcal{H}^{N-1} - \int_{\{|v-g| \le c_v\}} |v-g|^p\hat{h} d\mathcal{H}^{N-1} \\
 	&\ge \int_{\{|v-g| > c_v\}} c_v^p(h_v - \hat{h}) d\mathcal{H}^{N-1} - \int_{\{|v-g| \le c_v\}} c_v^p \hat{h} d\mathcal{H}^{N-1} \\
 	&= c_v^p \left( \int_{\partial\Omega} h_v d\mathcal{H}^{N-1} - \int_{\partial\Omega} \hat{h} d\mathcal{H}^{N-1} \right)= 0
 \end{align*}
 since both $h_v, \hat{h} \in \mathcal{H}_m(\partial\Omega)$.
 
 Therefore, we have proved that for $t\in (0,1)$ it holds
 $\psi'(t) \le 0$ and, as already discussed, this gives the optimality of $h_v$ which is achieved for $t=1$.

 Finally, we show that $h_v$ is the unique minimizer provided $v-g$ is not identically null on $\partial \Omega$. Let $h_1, h_2 \in \mathcal{H}_m(\partial\Omega)$ be two minimizers. Since the map $t \mapsto (1+\beta^{\frac{1}{p-1}}t)^{-(p-1)}$ is strictly convex for $t \ge 0$, it holds
 \[
 \frac{1}{2} \frac{|v-g|^p}{(1+\beta^{\frac{1}{p-1}}h_1)^{p-1}} + \frac{1}{2} \frac{|v-g|^p}{(1+\beta^{\frac{1}{p-1}}h_2)^{p-1}} \ge \frac{|v-g|^p}{\left(1+\beta^{\frac{1}{p-1}}\frac{h_1+h_2}{2}\right)^{p-1}},
 \]
 where the inequality is strict almost everywhere on the set $\{v \neq g\} \cap \{h_1 \neq h_2\}$. Since $\frac{h_1+h_2}{2} \in \mathcal{H}_m(\partial\Omega)$ and both $h_1$ and $h_2$ are minimizers, the integrals over $\partial \Omega$ of both sides must be equal, which forces $h_1 = h_2$ almost everywhere on $\{v \neq g\}$. 
 On the other hand, on the set $\{v = g\}$, any optimal configuration must vanish almost everywhere. On the contrary, if an optimal profile had positive mass on $\{v=g\}$, one could strictly decrease the integrand by moving this mass to the set $\{v \neq g\}$ (which has positive measure by assumption). Therefore, $h_1 = h_2 = 0$ on $\{v = g\}$. This implies $h_1 = h_2$ almost everywhere on $\partial \Omega$, proving uniqueness.
 \end{proof}
 
 Next result concerns convexity properties of the functional $J$ defined in \eqref{defJ}.
 
 \begin{proposition}\label{prop_convexity}
Let $\Omega$ be connected and let $p>1$, $\beta>0$, and $m>0$ be fixed. Finally let $f\in L^{\frac{p}{p-1}}(\Omega)$ and let $g\in L^p(\partial\Omega)$.
Then the functional $J(v,h)$ defined in \eqref{defJ} is convex in $W^{1,p}(\Omega) \times  \mathcal{H}_m(\partial\Omega)$ that is 
\[
\frac{1}{2}\left[ J(v_1, h_1) + J(v_2, h_2) \right] \ge J\left(\frac{v_1+v_2}{2}, \frac{h_1+h_2}{2}\right),
\]
where the equality holds if and only if there exists a constant $k \in \mathbb{R}$ such that $v_2 = v_1 + k$ almost everywhere in $\Omega$, and on $\partial\Omega$ it holds:
\begin{equation}\label{cond_conv}
\frac{v_1-g}{1+ \beta^{\frac{1}{p-1}}h_1} = \frac{v_1+k-g}{1+\beta^{\frac{1}{p-1}}h_2}.
\end{equation}
\end{proposition}

\begin{proof}
Let us analyze the convexity of each term of the functional $J$.

First observe that, as the function $x \mapsto |x|^p$ is strictly convex in $\mathbb{R}^N$ for $p>1$, for the first term of the functional $J$ it holds
\begin{equation}\label{convgrad}
	\frac{1}{2}\left(\frac{1}{p}\int_{\Omega}|\nabla v_1|^{p}dx + \frac{1}{p}\int_{\Omega}|\nabla v_2|^{p}dx\right) \ge \frac{1}{p}\int_{\Omega}\left|\frac{\nabla v_1+\nabla v_2}{2}\right|^{p}dx,
\end{equation}
where equality holds if and only if $\nabla v_1 = \nabla v_2$ almost everywhere in $\Omega$; as $\Omega$ is connected, this implies $v_2 = v_1 + k$ almost everywhere in $\Omega$ where $k \in \mathbb{R}$.

Now we focus on the boundary term. Let us set $\psi(v,h) = \frac{|v-g|^p}{(1+\beta^{\frac{1}{p-1}}h)^{p-1}}$ and $y = 1+ \beta^{\frac{1}{p-1}}h$. Then by denoting
\[
\psi(v,h) = \frac{|v-g|^p}{(1+\beta^{\frac{1}{p-1}}h)^{p-1}} = \frac{|v-g|^p}{y^{p-1}} = y\, \varphi\left(\frac{v-g}{y}\right),
\]
where $\varphi(t) = |t|^p$ and, as $\varphi$ is strictly convex for $p>1$, we can apply Lemma \ref{lemmaconvex} yielding to
\begin{equation}
	\label{gco}
	\frac{1}{2}\psi(v_1,h_1) + \frac{1}{2}\psi(v_2,h_2) \ge \psi\left(\frac{v_1+v_2}{2}, \frac{h_1+h_2}{2}\right).
\end{equation} 
The equality sign in \eqref{gco} is achieved if and only if the points $(v_1-g, y_1)$ and $(v_2-g, y_2)$ are collinear with the origin. This means that they must have the same ratio, which is equivalent to
\[ 
\frac{v_1-g}{1+\beta^{\frac{1}{p-1}}h_1} = \frac{v_2-g}{1+\beta^{\frac{1}{p-1}}h_2} \ \text{on }\partial\Omega.
\] 
By integrating over $\partial\Omega$ we have then proved that the boundary term is convex, namely 
\begin{equation}\label{convbordo}
	\begin{aligned}
	&\frac{1}{2}\int_{\partial\Omega}\frac{|v_1-g|^p}{(1+\beta^{\frac{1}{p-1}}h_1)^{p-1}}d\mathcal{H}^{N-1} + \frac{1}{2}\int_{\partial\Omega}\frac{|v_2-g|^p}{(1+\beta^{\frac{1}{p-1}}h_2)^{p-1}}d\mathcal{H}^{N-1} \\
	&\ge \frac{1}{2}\int_{\partial\Omega}\frac{|\frac{v_1+v_2}{2}-g|^p}{(1+\beta^{\frac{1}{p-1}}\frac{h_1+h_2}{2})^{p-1}}d\mathcal{H}^{N-1}.
	\end{aligned}
\end{equation}
By gathering together \eqref{convbordo}, \eqref{convgrad} and the fact that the term involving $f$ is linear in $v$, one yields to
  \begin{equation*}
    \label{tesipass}
\frac{1}{2}\left[ J(v_1, h_1) + J(v_2, h_2) \right] \ge J\left(\frac{v_1+v_2}{2}, \frac{h_1+h_2}{2}\right)
\end{equation*}
with equality sign which must hold simultaneously in the gradient term, which requires $v_2 = v_1 + k$, and in the boundary term, which means $(v_1-g)(1+ \beta^{\frac{1}{p-1}}h_2) = (v_1+k-g)(1+\beta^{\frac{1}{p-1}}h_1)$ for some $k\in \mathbb{R}$. 

\end{proof}

Next result concerns the coercivity of the functional $J$ in $W^{1,p}(\Omega)$.
\begin{proposition}\label{prop_coerc}
Let $p>1$, $\beta > 0$, and $m > 0$ be fixed. Finally let $f\in L^{\frac{p}{p-1}}(\Omega)$ and let $g\in L^p(\partial\Omega)$. Then there exist two positive constants $C_1$ and $C_2$ such that for any $(v,h) \in W^{1,p}(\Omega)\times\mathcal{H}_{m}(\partial\Omega)$, it holds
\[
J(v,h) \ge C_1 \|v\|_{W^{1,p}(\Omega)}^p - C_2.
\]
\end{proposition}

\begin{proof}
	Let us start by estimating from below the boundary term.
	It follows from an application of the H\"older inequality with exponents $p$ and $\frac{p}{p-1}$ that it holds
	\[ 
	\int_{\partial\Omega} |v-g| \,d\mathcal{H}^{N-1}  \le \left( \int_{\partial\Omega} \frac{|v-g|^p}{(1+\beta^{\frac{1}{p-1}}h)^{p-1}} \,d\mathcal{H}^{N-1} \right)^{\frac{1}{p}} \left( \int_{\partial\Omega} \left(1+\beta^{\frac{1}{p-1}}h\right) \,d\mathcal{H}^{N-1} \right)^{\frac{p-1}{p}}.
	\]
	Since for the second term on the right-hand of the previous, it holds
	\[
	\int_{\partial\Omega} \left(1+\beta^{\frac{1}{p-1}}h\right) \,d\mathcal{H}^{N-1}  = P(\Omega) + m\beta^{\frac{1}{p-1}},
	\]
	one clearly has that
	\begin{equation}\label{coerc_bound}
		\int_{\partial\Omega}\frac{|v-g|^{p}}{(1+\beta^{\frac{1}{p-1}}h)^{p-1}}d\mathcal{H}^{N-1} \ge \frac{\left(\int_{\partial\Omega}|v-g|\,d\mathcal{H}^{N-1}\right)^p}{\left(P(\Omega) + m\beta^{\frac{1}{p-1}}\right)^{p-1}}.
	\end{equation}
Moreover, using the triangle inequality $\int_{\partial\Omega} |v| d\mathcal{H}^{N-1} \le \int_{\partial\Omega} |v-g| d\mathcal{H}^{N-1}  + \int_{\partial\Omega} |g| d\mathcal{H}^{N-1} $ together with the convexity inequality $(A+B)^p \le 2^{p-1}(A^p + B^p)$ for any $A, B \ge 0$, it follows from \eqref{coerc_bound} that
	\[
	J(v,h) \ge \frac{1}{p}\int_{\Omega}|\nabla v|^{p}dx + \frac{\beta\left(\int_{\partial\Omega}|v|\,d\mathcal{H}^{N-1}\right)^p}{p 2^{p-1}\left(P(\Omega) + m\beta^{\frac{1}{p-1}}\right)^{p-1}} - C(g) - \int_{\Omega}fv dx,
	\]
	where $C(g)$ is a positive constant depending only on $g$, $\beta$, $m$, $p$, and $\Omega$.
	Then one can apply Theorem 4.4.6 of \cite{ziemer} to obtain that there exists a positive constant $C_0$ such that
	\[
	J(v,h) \ge C_0 \|v\|_{W^{1,p}(\Omega)}^p - C(g) - \int_{\Omega}fv dx.
	\]
	Now observe that for the term involving $f$, one can apply the H\"older, the Young and the Sobolev inequalities in order to obtain that for any $\varepsilon>0$
	\[
	J(v,h) \ge C_0 \|v\|_{W^{1,p}(\Omega)}^p - \frac{\varepsilon^p}{p}\|v\|_{W^{1,p}(\Omega)}^p - \frac{p-1}{p\varepsilon^{\frac{p}{p-1}}}C_S\|f\|_{L^{\frac{p}{p-1}}(\Omega)}^{\frac{p}{p-1}} - C(g),
	\]
	where $C_S$ is a constant depending on the Sobolev embedding. By choosing $\varepsilon>0$ small enough, we can absorb the norm of $v$ and conclude the proof. 
\end{proof}

We are now in position to prove Theorem \ref{teo_min}.

\begin{proof}[Proof of Theorem \ref{teo_min}]
	First observe that it follows from Proposition \ref{prop_bordo} that it is sufficient to minimize the functional $J(v, h_v)$ over $v \in W^{1,p}(\Omega)$; in particular, with an abuse of notation, we simply refer to it by $J(v)$.
	
	\medskip
	
	Proposition \ref{prop_coerc} gives that $J$ is bounded from below and coercive, so that we can assume that the existence of a minimizing sequence $u_n$ for $J(v)$ which is bounded in $W^{1,p}(\Omega)$ with respect to $n$.

	Then, by standard theory, this provides the existence of $u \in W^{1,p}(\Omega)$ such that, up to a subsequence, $u_n \rightharpoonup u$ in $W^{1,p}(\Omega)$ and such that $u_n \to u$ in $L^p(\Omega)$. The trace theorem also ensures strong convergence of $u_n \to u$ in $L^p(\partial\Omega)$.
	
	\medskip
	
	Now our aim is to prove that the functional $J$ is weakly lower semicontinuous, i.e., $\liminf_{n\to\infty} J(u_n) \ge J(u)$ whenever $u_n \rightharpoonup u$ in $W^{1,p}(\Omega)$. To this end, we will only focus on the boundary integral, as the weak lower semicontinuity of the remaining terms is standard.
	First observe that, as $u_n \to u$ in $L^p(\partial\Omega)$, it is sufficient to show that $c_{u_n} \to c_u$ so that $h_{u_n} \to h_u$ $\mathcal{H}^{N-1}-$ a.e. on $\partial\Omega$ and this would allow to pass to the limit with respect to $n$ in the boundary term.

First observe that $c_{u_n}$ is bounded; otherwise it would lead to a contradiction with the fact that the sequence of norms $\|u_n\|_{L^1(\partial\Omega)}$ is itself bounded and  $g\in L^p(\partial\Omega)$. Then there exists a convergent subsequence converging to some $\overline{c}$ which can be shown to be equal to $c_u$ by simply passing to the limit into \eqref{eq:c_definition} with respect to $n$.  Indeed, thanks to the strong convergence of $u_{n}$ in $L^p(\partial\Omega)$ and applying the Dominated Convergence Theorem, one can deduce, by uniqueness, that $\overline{c}=c_u$ by passing to the limit in the definition of $c_{u_n}$. 
This is actually sufficient to prove that $J$ is weak lower semicontinuous. This implies that $u$ is a minimizer for $J$ and the pair $(u, h_u)$ is a solution to \eqref{pbmin}. 
	
	\medskip
	
	It remains to prove the uniqueness of the minimizer. First observe that it is standard to show that any minimizer $u$ satisfies \eqref{eqpbmin} and that, under the assumption $\int_\Omega f \,dx \neq 0$, $c_u$ is strictly positive. Indeed, this follows by integrating \eqref{eqpbmin} over $\Omega$ implying that if $u \equiv g$ on $\partial\Omega$, then $\int_\Omega f \,dx$ would necessarily vanish. Consequently, by Proposition \ref{prop_bordo}, $c_u$ cannot be zero, which implies that any minimizer cannot be identically equal to $g$ on $\partial\Omega$.
	
	Now let us assume by contradiction that there exist two distinct global minimizers $(u_1, h_{u_1})$ and $(u_2, h_{u_2})$ for the functional $J$. 
	As Proposition \ref{prop_convexity} is in force, one gains that $(\bar{u}, \bar{h}) = \left(\frac{u_1+u_2}{2}, \frac{h_{u_1}+h_{u_2}}{2}\right)$ must also be a global minimizer. This implies that the equality case must hold in the convexity inequalities for each term of the functional and that \eqref{cond_conv} needs to hold, namely
	\begin{equation}\label{cond_convintotheorem}
		\frac{u_1-g}{1+ \beta^{\frac{1}{p-1}}h_{u_1}} = \frac{u_1+k-g}{1+\beta^{\frac{1}{p-1}}h_{u_2}} \quad \text{on } \partial\Omega,
	\end{equation}
	where $u_2=u_1+k$ in $\Omega$ for some $k\in\mathbb{R}$.
	
	Now let assume by contradiction that $k \neq 0$.
	
	Then let $\gamma = \beta^{\frac{1}{p-1}}$ and define $v_i = u_i - g$. In particular \eqref{cond_convintotheorem} reads as
	\begin{equation}\label{cond_convintotheorem2}
		\frac{v_1}{1+\gamma h_{u_1}} = \frac{v_2}{1+\gamma h_{u_2}} \quad \text{on } \partial\Omega.
	\end{equation}
	Let us observe that, since $1+\gamma h_{u_i} \ge 1 > 0$, one immediately deduces from \eqref{cond_convintotheorem2} that $v_1$ and $v_2$ must share the same sign almost everywhere.
	
	Moreover also recall that the optimal thickness gives that
	\begin{equation}\label{opthintoth}
		h_{u_i} = 0 \quad \text{if } |v_i| \le c_{u_i}, \qquad 1+\gamma h_{u_i} = \frac{|v_i|}{c_{u_i}} \quad \text{if } |v_i| > c_{u_i}.
	\end{equation}
	As we want to evaluate
	\[\int_{\partial\Omega} \frac{|v_2|^p}{(1+\gamma h_{u_2})^{p-1}}d \mathcal{H}^{N-1} - \int_{\partial\Omega} \frac{|v_1|^p}{(1+\gamma h_{u_1})^{p-1}}d \mathcal{H}^{N-1}\]
	we set $E_i = \frac{|v_i|^p}{(1+\gamma h_{u_i})^{p-1}}$ and we split the boundary $\partial\Omega$ into four disjoint subsets and apply \eqref{cond_convintotheorem} and \eqref{opthintoth}.
	Let us consider the following cases:
	
	\textbf{Case 1: on \{$|v_1| \le c_{u_1}$, $|v_2| \le c_{u_2}$\}.}
	Here $h_{u_1} = 0$ and $h_{u_2} = 0$. Condition \eqref{cond_convintotheorem2} simplifies to $v_1 = v_2$. Since $v_2 = v_1 + k$, this implies $k = 0$, contradicting our assumption. Thus, if $k\not =0$, this set must have zero measure.
	
	\textbf{Case 2: on \{$|v_1| \le c_{u_1}$, $|v_2| > c_{u_2}$\}.}
	Here $h_{u_1} = 0$ and $1+\gamma h_{u_2} = \frac{|v_2|}{c_{u_2}}$. Then condition \eqref{cond_convintotheorem2} yields
	\[v_1 = \frac{v_2}{|v_2|/c_{u_2}},\]
	which gives $v_1 = c_{u_2} \frac{v_2}{|v_2|}$ so that $|v_1| = c_{u_2}$. Since we are in the zone where $|v_1| \le c_{u_1}$, this implies $c_{u_2} \le c_{u_1}$.
	Moreover $E_1 = |v_1|^p = c_{u_2}^p$ and $E_2 = \frac{|v_2|^p}{(|v_2|/c_{u_2})^{p-1}} = c_{u_2}^p \frac{|v_2|}{c_{u_2}} = c_{u_2}^p(1+\gamma h_{u_2})$. Therefore one gets
	\[E_2 - E_1 = c_{u_2}^p(1+\gamma h_{u_2}) - c_{u_2}^p = c_{u_2}^p \gamma h_{u_2} =  c_{u_2}^p \gamma (h_{u_2} - h_{u_1}),\]
	as $h_{u_1} = 0$.
	
	\textbf{Case 3: on \{$|v_1| > c_{u_1}$, $|v_2| \le c_{u_2}$\}.}
	 By symmetry with Case 2, one has $|v_2| = c_{u_1}$. Since $|v_2| \le c_{u_2}$, we deduce $c_{u_1} \le c_{u_2}$. 
	Then $E_1 = c_{u_1}^p(1+\gamma h_{u_1})$ and $E_2 = c_{u_1}^p$. Their difference is:
	\[E_2 - E_1 = c_{u_1}^p - c_{u_1}^p(1+\gamma h_{u_1}) = -c_{u_1}^p \gamma h_{u_1} = c_{u_1}^p \gamma (h_{u_2} - h_{u_1}),\]
	as $h_{u_2} = 0$.
	
	\textbf{Case 4: on \{$|v_1| > c_{u_1}$, $|v_2| > c_{u_2}$\}.}
	Here $1+\gamma h_{u_1} = \frac{|v_1|}{c_{u_1}}$ and $1+\gamma h_{u_2} = \frac{|v_2|}{c_{u_2}}$ and condition \eqref{cond_convintotheorem2} gives
	\[\frac{v_1}{|v_1|/c_{u_1}} = \frac{v_2}{|v_2|/c_{u_2}},\]
	which simplifies to $c_{u_1} \frac{v_1}{|v_1|} = c_{u_2} \frac{v_2}{|v_2|}$. As $v_1$ and $v_2$ have the same sign, this implies $c_{u_1} = c_{u_2}$ and we denote this value as $c$. 
	Therefore $E_i = \frac{|v_i|^p}{(|v_i|/c)^{p-1}} = c^p \frac{|v_i|}{c} = c^p (1+\gamma h_{u_i})$. The difference is simply
	\[E_2 - E_1 = c^p \gamma (h_{u_2} - h_{u_1}).\]
	
	\medskip
	
	We are now ready to conclude the proof by exploiting the contradiction.
	
	First, observe that $h_{u_1}$ and $h_{u_2}$ cannot be identically zero on $\partial\Omega$ since they both have $L^1$ mass equal to $m > 0$. If $c_{u_1} < c_{u_2}$, Case 2 would lead to the contradiction $c_{u_2} \le c_{u_1}$, meaning Case 2 must have zero measure. This implies $h_{u_2}$ must be entirely supported on Case 4, forcing Case 4 to have positive measure. But Case 4 implies $c_{u_1} = c_{u_2}$, giving a contradiction. A symmetric argument rules out $c_{u_2} < c_{u_1}$. Therefore, we necessarily have $c_{u_1} = c_{u_2} = c$. 
	
	With $c_{u_1} = c_{u_2} = c$, all the possible cases (including Case 1) consistently yield that almost everywhere on $\partial\Omega$ it holds
	\[E_2 - E_1 = c^p \gamma \big( h_{u_2} - h_{u_1} \big).\]
	
	Since $(u_1, h_{u_1})$ and $(u_2, h_{u_2})$ are both global minimizers, it holds that
	\[J(u_2) - J(u_1) = 0.\]
	Noting that $\nabla u_2 = \nabla u_1$, substituting $u_2 - u_1 = k$, and integrating the relation for $E_2 - E_1$, we get
	\[\frac{\beta}{p} c^p \gamma \int_{\partial\Omega} (h_{u_2} - h_{u_1}) \, d\mathcal{H}^{N-1} - k \int_{\Omega} f \, dx = 0.\]
	Since both profiles have total mass $m$, we have $\int_{\partial\Omega} h_{u_1} d\mathcal{H}^{N-1} = \int_{\partial\Omega} h_{u_2} d\mathcal{H}^{N-1} = m$, and the boundary integral vanishes completely.  This implies
	\[-k \int_{\Omega} f \, dx = 0,\]
which forces $k = 0$ as we assumed $\int_\Omega f dx \neq 0$; in particular this contradicts our initial assumption that $k \neq 0$.
	
	Therefore, we must have $k=0$, yielding $u_1 = u_2$ in $\Omega$. This also directly implies $h_{u_1} = h_{u_2}$ on $\partial\Omega$, concluding the proof of uniqueness.
\end{proof}

We conclude the section with an estimate of the heat content in case the datum $f$ is equal to $1$ and $g$ is constant.  In particular we also show that, for a constant insulator, the ball maximizes the heat content. 
Let us stress that in the sequel $\omega_N$ denotes the measure of the unit ball in $\mathbb{R}^N$.	
 \begin{theorem}\label{teo_heatcontent}
Let $p>1$, $m>0$, $\beta>0$ and let $f \equiv 1$ and $g \equiv \overline{g}\in \mathbb{R}$. Let $(u,h)$ be a minimizer of $J$ defined in \eqref{defJ}. 
Then it holds
\[
	\int_{\Omega}u\,dx \le \overline{g}|\Omega| + \left(N\omega_{N}^{\frac{1}{N}}\right)^{-\frac{p}{p-1}} \left[ \frac{(p-1)N}{p+(p-1)N}|\Omega|^{1+\frac{p}{(p-1)N}} + |\Omega|^{\frac{p}{(p-1)N}}\left(\frac{P(\Omega)}{\beta^{\frac{1}{p-1}}}+m\right) \right].
\]
The equality in the estimate holds if $\Omega$ is a ball and the insulator $h$ is constant.
 \end{theorem}
 \begin{proof}
	First observe that, as a consequence of the strong maximum principle, one has $u> \overline{g}$ on $\overline{\Omega}$. For $t>\overline{g}$ and $\delta>0$ let us define
 	\[
 	\varphi(x) = \begin{cases} 0 & \text{if } u(x) \le t, 
 	\\ u(x)-t & \text{if } t < u(x) < t+\delta, 
 	\\ \delta & \text{if } u(x) \ge t+\delta, \end{cases}
 	\]
 	which we take as a test function in the weak formulation of \eqref{eqpbmin}; this yields to
 	\[
 	\int_{\{t<u<t+\delta\}} |\nabla u|^p dx + \beta\int_{\partial\Omega} \frac{(u-\overline{g})^{p-1}\varphi}{(1+\beta^{\frac{1}{p-1}}h)^{p-1}} \,d\mathcal{H}^{N-1} = \int_{\{u>t\}} \varphi \,dx.
 	\]
 	Then one can divide the entire expression by $\delta$ and take the limit as $\delta \to 0^+$. By using also the co-area formula, one gets that for almost every $t>\bar g$ it holds
 	\begin{equation}\label{eqmu}
 	\mu(t) = \int_{\{u=t\}} |\nabla u|^{p-1} d\mathcal{H}^{N-1} + \beta\int_{\partial\Omega \cap \{u>t\}} \frac{(u-\overline{g})^{p-1}}{\left(1+\beta^{\frac{1}{p-1}}h\right)^{p-1}} d\mathcal{H}^{N-1}
 	\end{equation}
 	where $\mu(t) := |\{x \in \Omega : u(x) > t\}|$.
 	
 	Now let $P(t) := P(\{u>t\})$. By coarea formula for Sobolev functions, for a.e. $t>\bar g$, $ P(t)= \int_{\{u=t\}} 1 \,d\mathcal{H}^{N-1}+ \int_{\partial \Omega \cap \{u>t\}} 1 \,d\mathcal{H}^{N-1}$. Also observe that here $u$ is identified with its precise representative (and with its trace on $\partial\Omega$); actually, this distinction is unnecessary, since standard regularity results yield $u\in C(\overline\Omega)$.
	 Applying the generalized H\"older inequality (see \cite[Proposition $A.1$]{DOS}), we obtain
 	\begin{equation*}
 		\begin{aligned}
 		P(t)^p \le  &\left( \int_{\{u=t\}} |\nabla u|^{p-1}d\mathcal{H}^{N-1} +\beta \int_{\partial\Omega \cap\{u>t\}} \frac{(u-\overline{g})^{p-1}}{(1+\beta^{\frac{1}{p-1}}h)^{p-1}}d\mathcal{H}^{N-1}\right) \\
 		&\times \left( \int_{\{u=t\}} \frac{1}{|\nabla u|}d\mathcal{H}^{N-1} + \int_{\partial\Omega \cap\{u>t\}} \frac{(1+\beta^{\frac{1}{p-1}}h)}{\beta^{\frac{1}{p-1}}(u-\overline{g})} d\mathcal{H}^{N-1} \right)^{p-1}.
 		\end{aligned} 
 	\end{equation*}
 	Now observe that it follows from \eqref{eqmu} and from an application of the co-area formula that
 	 \begin{equation*}
 		\begin{aligned}
 			P(t)^p \le  \mu(t) \left(- \mu'(t)  + \int_{\partial\Omega \cap\{u>t\}} \frac{(1+\beta^{\frac{1}{p-1}}h)}{\beta^{\frac{1}{p-1}}(u-\overline{g})}d\mathcal{H}^{N-1}  \right)^{p-1}.
 		\end{aligned} 
 	\end{equation*}
 	In particular, recalling that $P(t) \ge N\omega_N^{\frac{1}{N}} \mu(t)^{\frac{N-1}{N}}$, the previous inequality leads to 
 	 \[
 	 \mu(t) \le (N\omega_N^{\frac{1}{N}})^{-\frac{p}{p-1}} \mu(t)^{\frac{p}{(p-1)N}}\left(-\mu'(t) + \int_{\partial\Omega \cap\{u>t\}} \frac{(1+\beta^{\frac{1}{p-1}}h)}{\beta^{\frac{1}{p-1}}(u-\overline{g})}d\mathcal{H}^{N-1}\right).
 	 \]
 	Now, by integrating with respect to $t$ and using that $\mu(t)\le |\Omega|$, one gets
 	\[\int_\Omega (u-\overline{g}) \ dx \le (N\omega_N^{\frac{1}{N}})^{-\frac{p}{p-1}} \left(\frac{(p-1)N}{p+(p-1)N}|\Omega|^{\frac{p}{(p-1)N}+1} + |\Omega|^{\frac{p}{(p-1)N}}\left(\frac{P(\Omega)}{\beta^{\frac{1}{p-1}}}+m\right) \right).
 	\]
 	Equality holds when $\Omega$ is a ball and $h$ is constant, and it follows from the fact that $u$ is actually radial.
 	This concludes the proof.
 \end{proof}

\section{On a concentration breaking phenomenon}

In this section we deal with a concentration breaking phenomenon involving the best insulating configuration $h$, which is defined in \eqref{defh} and which we recall here for simplicity
\begin{equation}\label{h_conc}
	h_v(\sigma) := \begin{cases}
	\displaystyle \frac{|v(\sigma)-g(\sigma)|}{c_{v}\beta^{\frac{1}{p-1}}} - \frac{1}{\beta^{\frac{1}{p-1}}} & \text{if } |v(\sigma)-g(\sigma)| > c_v, \\
	0 & \text{otherwise,}
\end{cases}
\end{equation}
where $c_v \ge 0$ is the unique constant satisfying
\begin{equation*}
c_v = \left( \frac{1}{|\{ |v-g| > c_v \}| + m\beta^{\frac{1}{p-1}}} \right) \int_{\{ |v-g| > c_v \}} |v-g| \, d\mathcal{H}^{N-1}.
\end{equation*}
Let us recall that, from Theorem \ref{teo_min}, there exists  
a pair $(u,h_u) \in W^{1,p}(\Omega) \times \mathcal{H}_m(\partial\Omega)$ minimizing \eqref{pbmin} which is unique if $\Omega$ is connected and $\int_\Omega f dx \neq 0$. 
Also observe that the minimizing pair $(u,h_u)$ is also a solution to \eqref{eqpbmin}.

\medskip

In this section we establish the existence of a critical mass threshold $\overline{m}>0$ such that when the total mass of the insulating material $m$ is below this threshold, the optimal configuration does not cover the entire boundary. Instead, the material concentrates exclusively on the regions with the highest thermal stress to maximize efficiency, leaving parts of the boundary exposed ($h=0$). The following result is proved in the model case $f=1$ and for $g$ sufficiently regular. Let us finally stress that the hypotheses required in Theorem \ref{teo_breaking} are further developed in Corollaries \ref{cor_nonpalla} and \ref{cor_palla_noncostante} below. In particular, we have concentration breaking either if the domain is not a ball and the external temperature profile $g$ is constant, or the domain is a ball and $g$ is not constant.

\begin{theorem}\label{teo_breaking}
	Let $\Omega$ be a connected, smooth, bounded and open set of $\mathbb{R}^N$.  Moreover let $f \equiv 1$, $g \in C^0(\partial\Omega)$ and $m,\beta > 0$. Assume that either $\partial\Omega$ is connected or $g$ is constant. Let $u_0$ be the unique solution to the normalized problem
	\begin{equation} \label{eq:u0}
		\begin{cases}
			-\Delta_p u_0 = 1 & \text{in } \Omega, \\
			|\nabla u_0|^{p-2}\dfrac{\partial u_0}{\partial\nu} = -\dfrac{|\Omega|}{P(\Omega)} & \text{on } \partial\Omega, \\
			\displaystyle\int_{\partial\Omega} u_0 \, d\mathcal{H}^{N-1} = 0. & 
		\end{cases}
	\end{equation}
	Let $(u,h_u)$ be the unique solution to \eqref{pbmin} where $h_u$ is defined by \eqref{h_conc}.
	Assume that the function $\sigma \mapsto u_0(\sigma) - g(\sigma)$ is not constant on $\partial\Omega$.
	Then, there exists $\overline{m}>0$ such that:
	\begin{itemize}
		\item[i)] if $m < \overline{m}$ then $h_u$ vanishes on a subset of $\partial\Omega$ of positive $\mathcal H^{N-1}-$ measure;
				\item[ii)] if $m > \overline{m}$ then $h_u>0$ on $\partial\Omega$.
	\end{itemize}
	Moreover, if $m=\overline{m}$, 
 $\{h_{u}=0\}\neq\emptyset$ although it may fail to have positive
$\mathcal H^{N-1}$-measure. 
\end{theorem}
\begin{proof}
	Let us first prove part \textit{i)}. We assume by contradiction that the optimal distribution $h_{u}(\sigma) > 0$ for almost every $\sigma \in \partial\Omega$; this clearly means $|u(\sigma)-g(\sigma)| > c_u$ almost everywhere on $\partial\Omega$. In particular, recalling \eqref{h_conc}, $u$ satisfies the following Neumann problem
		\begin{equation}\label{cv_conc-1}
		\begin{cases}
			-\Delta_p u = 1 & \text{in } \Omega, \\
			|\nabla u|^{p-2}\dfrac{\partial u}{\partial\nu} = -\beta c_u^{p-1} \operatorname{sign}(u-g) & \text{on } \partial\Omega.
		\end{cases}
	\end{equation}

Then, since the right-hand side of the boundary condition is in $L^\infty(\partial\Omega)$, standard regularity results for quasilinear Neumann problems ensure that $u \in C(\overline{\Omega})$. Now observe that both $u$ and $g$ are continuous on $\partial\Omega$. Since we assumed $|u-g| > c_u$ almost everywhere, we want to show that $u(\sigma)-g(\sigma) > c_u$ for almost every $\sigma \in \partial\Omega$.
	First, if $\partial\Omega$ is connected, the continuous function $u-g$ cannot change sign. Thus, it holds either $u(\sigma)-g(\sigma) > c_u$ almost everywhere on $\partial\Omega$, or $u(\sigma)-g(\sigma) < -c_u$ almost everywhere on $\partial\Omega$. In particular, an integration by parts on \eqref{cv_conc-1} shows that $u(\sigma)-g(\sigma) < -c_u$ almost everywhere on $\partial\Omega$ cannot be the case. 
	On the other hand, if $g$ is constant, the strong maximum principle ensures that the solution $u$ to \eqref{cv_conc-1} cannot attain its global minimum on any boundary component where the outer normal derivative is strictly positive. This a priori excludes that $u-g < -c_u$ on any component of $\partial\Omega$. 
	Therefore, in either case, one has $u(\sigma)-g(\sigma) > c_u$ for almost every $\sigma \in \partial\Omega$ and the previous problem can be read as 
	\begin{equation*}
	\begin{cases}
		-\Delta_p u = 1 & \text{in } \Omega, \\
		|\nabla u|^{p-2}\dfrac{\partial u}{\partial\nu} =-\dfrac{|\Omega|}{P(\Omega)} \quad &\text{on } \partial\Omega,
	\end{cases}
	\end{equation*}
	where last equality on the boundary equation follows from \eqref{cv_conc-1} by taking $\varphi=1$ as a test function 
	\begin{equation}\label{cv_conc}
		\beta c_u^{p-1} P(\Omega)= -\int_{\partial\Omega} |\nabla u|^{p-2}\frac{\partial u}{\partial \nu} \, d\mathcal{H}^{N-1} =  \int_{\Omega} 1 \, dx = |\Omega|. 
	\end{equation}
	 Recalling the definition of $u_0$ in \eqref{eq:u0}, it holds that $u(x) = u_0(x) + K_m$ where $K_m$ is a constant related to the total mass of insulating $m$ and of $g$.  Indeed,  using that $\int_{\partial\Omega} h_{u} \, d\mathcal{H}^{N-1} = m$, one has
	\[
	m = \int_{\partial\Omega} \left( \dfrac{u-g}{c_u\beta^{\frac{1}{p-1}}} - \dfrac{1}{\beta^{\frac{1}{p-1}}} \right) d\mathcal{H}^{N-1} = \int_{\partial\Omega} \left( \dfrac{u_0 + K_m - g}{c_u\beta^{\frac{1}{p-1}}} - \dfrac{1}{\beta^{\frac{1}{p-1}}} \right) d\mathcal{H}^{N-1},
	\]
	and, as the normalization $\int_{\partial\Omega} u_0 d\mathcal{H}^{N-1} = 0$ is in force and $c_u = \left(\frac{|\Omega|}{\beta P(\Omega)}\right)^{\frac{1}{p-1}}$ from \eqref{cv_conc}, this yields to
	\[
	K_m = \left( \frac{|\Omega|}{\beta P(\Omega)} \right)^{\frac{1}{p-1}} \left( 1 + \frac{m\beta^{\frac{1}{p-1}}}{P(\Omega)} \right) + \overline{g}, \quad \text{where } \overline{g} = \frac{1}{P(\Omega)}\int_{\partial\Omega} g \, d\mathcal{H}^{N-1}.
	\]
	Now, as $h_u>0$ almost everywhere on $\partial\Omega$, then
$u-g>c_u$ almost everywhere on $\partial\Omega$. Hence, by continuity, one gets
	\[
	u_{0}(\sigma)+K_{m}-g(\sigma) \ge c_{u},\quad \forall \sigma\in\partial\Omega;
	\]
	that is
	\[
	-\delta:=\min_{\sigma\in\partial\Omega}\left[ u_{0}(\sigma) - g(\sigma) + \overline{g} \right] \ge c_u-K_m+\overline{g}= -\dfrac{\beta^{\frac{1}{p-1}} c_{u}}{P(\Omega)}m,
	\]
	and $\delta$ is strictly positive. Indeed, by assumption the function $u_0 - g$ is not constant on $\partial\Omega$. Since the average of $u_{0}-g+\bar g$ on $\partial\Omega$ vanishes (recalling that $\int_{\partial\Omega} u_0 \, d\mathcal{H}^{N-1} = 0$), its minimum must be strictly negative.
	In particular, the previous inequality gives that
	\[
	m \ge \dfrac{\delta P(\Omega)}{\beta^{\frac{1}{p-1}} c_{u}}=\delta \dfrac{P^\frac{p}{p-1}(\Omega)}{|\Omega|^{\frac{1}{p-1}}}=:\overline{m},
	\]
	providing that, if $m < \overline{m}$, then $h_{u}$ must vanish in some part of $\partial\Omega$ with positive measure.  This concludes the proof of part i).
	
	\medskip
	
	In order to complete the proof, we are left to show that, if $m > \overline{m}$, then $h_u > 0$ for all $\sigma \in \de\Omega$. 
	To this aim, we construct a candidate pair $(v, h_v)$ whose insulator covers the entire boundary, which, as we have already seen, satisfies a standard Neumann problem. Then, we verify that this candidate is admissible  and that it is the unique global minimizer to \eqref{pbmin}.
	
	Then let us consider the function $v = u_0 + K_m$, where $u_0$ and $K_m$ are, respectively, the function and the constant defined in the first part of the proof.
	First we show that the corresponding insulation density
	\[
	h_v(\sigma) = \dfrac{v(\sigma)-g(\sigma)}{c_v\beta^{\frac{1}{p-1}}} - \dfrac{1}{\beta^{\frac{1}{p-1}}}, \quad \text{with } c_v = \left(\dfrac{|\Omega|}{\beta P(\Omega)}\right)^{\frac{1}{p-1}},
	\]
	belongs to the admissible set $\mathcal{H}_m(\partial\Omega)$.
	The condition $h_v > 0$ is equivalent to $v-g > c_v$. As shown in the previous derivation, this inequality holds if and only if $m > \overline{m}$.

	Moreover, $h_v$ satisfies the integral constraint $\int_{\partial\Omega} h_v \, d\mathcal{H}^{N-1} = m$, as the constant $K_m$ has been fixed to ensure this equality.
	
	Now we focus on showing that the pair $(v,h_v)$ is a solution to \eqref{pbmin}.
	First observe that,
	by construction, $v$ satisfies $-\Delta_p v = 1$ in $\Omega$ and the boundary condition $|\nabla v|^{p-2}\partial_\nu v = -\beta c_v^{p-1}$.

	Therefore, the pair $(v, h_v)$ is a solution to \eqref{eqpbmin} and thus it is a minimizer for $J$ defined in \eqref{defJ}.
	In particular, by Proposition \ref{prop_convexity} and reasoning as in the proof of Theorem \ref{teo_min}, one can show that $(v, h_v)$ is the unique global minimizer for problem \eqref{pbmin}. 
	
	It remains to deal with the case $m=\overline{m}$. In this case one can assume by contradiction that
	$h_u>0$ everywhere on $\partial\Omega$. Then, as the
	minimum defining $\delta$ is again attained, in this case the inequality obtained in part $i)$ holds
	strictly
	which leads to $m>\overline{m}$, a contradiction. Hence $\{h_u=0\}\neq\emptyset$. 	
	This concludes the proof.
	\end{proof}

\begin{remark}
We observe that the value of $\overline m$ in Theorem \ref{teo_breaking} is independent of $\beta$. \triang
\end{remark}

The following corollaries detail the consequences of Theorem \ref{teo_breaking} by analyzing the geometry of $\Omega$ and the profile of $g$.

\begin{corollary}\label{cor_nonpalla}
	Let $\Omega$ be a connected, smooth, bounded and open set of $\mathbb{R}^N$ which is not a ball, and let $f \equiv 1$. Moreover let $m,\beta > 0$ and let $g$ be constant. Then i) and ii) of the statement of Theorem \ref{teo_breaking} hold true.
\end{corollary}
\begin{proof}
	Since $g$ is constant, the connectedness of $\partial\Omega$ is not required in order to apply Theorem \ref{teo_breaking}. We just need to verify that $u_0 - g$ is not constant on $\partial\Omega$, which reduces to requiring that $u_0$ is not constant on the boundary. Since $u_0$ is the solution to the torsional creep problem \eqref{eq:u0} with constant Neumann data, Serrin's overdetermined theorem \cite{fgk} ensures that $u_0$ can be constant on $\partial\Omega$ if and only if $\Omega$ is a ball. Since $\Omega$ is not a ball, $u_0$ is not constant, and the conclusion immediately follows from Theorem \ref{teo_breaking}.
\end{proof}

\begin{corollary}\label{cor_palla_noncostante}
	Let $\Omega$ be a ball and let $f \equiv 1$. Let $m,\beta > 0$ and let $g \in C(\partial\Omega)$ be non-constant. Then i) and ii) of the statement of Theorem \ref{teo_breaking} hold true.
\end{corollary}
\begin{proof}
	Since $\Omega$ is a ball, $\partial\Omega$ is connected. The auxiliary solution $u_0$ to \eqref{eq:u0} vanishes identically on $\partial\Omega$. Therefore, $u_0 - g = -g$ on $\partial\Omega$. Since $g$ is not constant by assumption, $u_0 - g$ is not constant on $\partial\Omega$. The conclusion follows from Theorem \ref{teo_breaking}.
\end{proof}

\begin{remark}
	From a physical perspective, Corollary \ref{cor_palla_noncostante} implies that even on a perfectly symmetric domain like a ball, if the external environment has temperature fluctuations and the amount of available insulation is small, it is optimal to concentrate all the insulating material exclusively on the colder regions of the boundary (where heat flux is higher), leaving the warmer parts completely exposed. \triang
\end{remark}

\begin{proposition}\label{prop_palla_costante}
	Let $\Omega$ be a ball and let $f \equiv 1$. Let $m,\beta > 0$ and let $g$ be constant. Then for any mass $m>0$, the optimal insulating configuration $h_u$ is strictly positive and covers the entire boundary $\partial\Omega$.
\end{proposition}
\begin{proof}
	In this symmetric setting, the auxiliary solution $u_0$ to \eqref{eq:u0} vanishes identically on $\partial\Omega$, and $g$ is constant. Consequently, the minimum $\delta$ defined in the proof of Theorem \ref{teo_breaking} as $-\delta = \min_{\partial\Omega} [u_0 - g + \overline{g}]$ is exactly zero. This implies that the critical mass $\overline{m}$ defined as $\delta P^\frac{p}{p-1}(\Omega) / |\Omega|^{\frac{1}{p-1}}$ is zero. Therefore, for any strictly positive amount of insulating material $m>0$, we are in the regime $m > \overline{m} = 0$, ensuring $h_u > 0$ on the whole boundary.
\end{proof}

\subsection{Examples and remarks}
\label{sec:disconnected_counterexample}  

This section is devoted to exploring the sharp nature of the assumptions in Theorem \ref{teo_breaking}. We begin by discussing the connectedness assumption of $\Omega$.

\begin{remark}
The assumption of connectedness of $\Omega$ in Theorem \ref{teo_breaking} is natural. Indeed, when $p=2$, $g=0$ and $\beta=+\infty$, in \cite[Example 2.4]{BBN17} the authors show that concentration breaking holds for any $m>0$ when $\Omega$ is the union of two disjoint balls of different radii.
\triang
\end{remark}

Next, we investigate the necessity of the non-degeneracy condition, which requires $u_0 - g$ to be non-constant on the boundary $\partial\Omega$. 
To illustrate its critical role, we present an explicit example showing how the threshold $\overline{m}$ collapses when this assumption is violated.

\begin{example}
	Let $\Omega$ be an ellipse, and let $u_0$ be the solution to \eqref{eq:u0}. Since $\Omega$ is not a ball, Serrin's overdetermined theorem guarantees that $u_0$ cannot be constant on the boundary $\partial\Omega$. 
	However, if we fix $g(\sigma) = u_0(\sigma)$ for all $\sigma \in \partial\Omega$, we obtain $u_0 - g \equiv 0$.
	In this scenario, the parameter $\delta$ introduced in the proof of Theorem \ref{teo_breaking} vanishes. As a direct consequence, the critical mass threshold collapses to zero ($\overline{m} = 0$). \triang  	
\end{example}
In particular, the above example	shows that, for any strictly positive amount of insulating material $m > 0$, the optimal thickness $h_u$ remains strictly positive and uniform along the entire boundary. The concentration breaking phenomenon never occurs, proving that a non-constant $u_0-g$ is strictly necessary to have concentration breaking.

\medskip

Next we discuss the hypothesis ``$\partial\Omega$ connected or $g$ constant'' given in Theorem \ref{teo_breaking} which, in general, is a necessary assumption to the validity of the result.

Before constructing an explicit example, let us explain how Theorem \ref{teo_breaking} produces a single threshold which gives full isolation or concentration breaking.

Suppose the optimal layer covers the whole boundary, namely $h_u>0$ on $\partial\Omega$ which means \ $|u-g|>c_u$ there. By \eqref{defh} the solution $u$, when $p=2$ and $\beta=1$, solves the Neumann problem
\begin{equation}\label{eq:rem_neumann}
-\Delta u=1 \ \text{ in }\Omega,\qquad \partial_\nu u=-c_u\operatorname{sign}(u-g)\ \text{ on }\partial\Omega.
\end{equation}
Now observe that, in the proof of Theorem \ref{teo_breaking} the hypothesis ``$\partial\Omega$ connected or $g$ constant'' is used just to force the sign of $u-g$ to be the same on the whole boundary, namely $u-g>c_u$ everywhere. Indeed, if $\partial\Omega$ is connected the continuous function $u-g$ cannot change sign, while if $g$ is constant the strong maximum principle forbids $u-g<-c_u$ on any boundary component; in either case the globally negative sign is excluded by the divergence theorem. Once the sign is uniform, $\operatorname{sign}(u-g)\equiv 1$ and \eqref{eq:rem_neumann} becomes exactly the problem \eqref{eq:u0} defining $u_0$. Hence $u=u_0+K_m$ is a family of solutions depending on $m$, and ``$h_u>0$ everywhere'' reduces to the inequality $K_m>c_u-\min_{\partial\Omega}(u_0-g)$, which gives the critical mass $\overline m$.

In particular the mentioned hypothesis thus plays exactly one role: it ensures a unique admissible sign pattern for $u-g$. 
As we will see in the next example, when $\partial\Omega$ is disconnected and a non-constant $g$ is present, different sign patterns may be possible. 
Let us explicitly construct the example we have in mind to show that a second sign pattern survives and generates a different situation with respect to the one of Theorem \ref{teo_breaking}.

\medskip
\begin{example}\label{exane}
For $m>0$ let us explicitly construct the unique solution to  \eqref{pbmin} where $p=2$, $\beta=1$, $f\equiv 1$, and
\[
\Omega=\Big\{x\in\mathbb{R}^2:\ \tfrac12<|x|<1\Big\},\qquad
\Gamma_{1/2}:=\{|x|=\tfrac12\},\qquad \Gamma_1:=\{|x|=1\},
\]
so that $\partial\Omega=\Gamma_{1/2}\cup\Gamma_1$ has two connected components and
\[
|\Omega|=\tfrac{3\pi}{4},\qquad \mathcal{H}^1(\Gamma_{1/2})=\pi,\qquad \mathcal{H}^1(\Gamma_1)=2\pi,\qquad P(\Omega)=3\pi.
\]
The external temperature $g\in C(\partial\Omega)$ is piecewise constant on the two components,
\[
g\equiv M_{1/2} \ \text{ on } \Gamma_{1/2},\qquad g\equiv M_{1} \ \text{ on } \Gamma_1,\qquad
s:=M_{1/2}-M_{1},
\]
where $M_{1/2},M_{1} \in \mathbb{R}$.
Under the previous set of assumptions we then assume that $h_u$ is strictly positive on $\partial\Omega$
so that $u$ satisfies \eqref{eq:rem_neumann}.

\smallskip
\textit{Step 1. The admissible sign patterns.}
First observe that, by standard boundary regularity for the $p$-Laplacian, $u\in C(\overline\Omega)$.
Hence, if $h_u>0$ on all of $\partial\Omega$, it follows from \eqref{h_conc} that the sign of $u-g$ is
constant on each component; then we define
\[
\varepsilon_{1/2}:=\operatorname{sign}(u-g)\big|_{\Gamma_{1/2}}\in\{\pm1\},\qquad
\varepsilon_1:=\operatorname{sign}(u-g)\big|_{\Gamma_1}\in\{\pm1\},
\]
and, by \eqref{eq:rem_neumann}, $\de_\nu u=-c_u\varepsilon_{1/2}$ on $\Gamma_{1/2}$ and
$\de_\nu u=-c_u\varepsilon_1$ on $\Gamma_1$ ($\nu$ denotes the outer unit normal of $\Omega$).
Taking $\varphi\equiv 1$ as a test function in \eqref{eq:rem_neumann}, one gets the compatibility
condition
\[
|\Omega|=c_u\big(\varepsilon_{1/2}\,\mathcal{H}^1(\Gamma_{1/2})+\varepsilon_1\,\mathcal{H}^1(\Gamma_1)\big)
=c_u\,\pi\,(\varepsilon_{1/2}+2\varepsilon_1),
\]
that is, as $|\Omega|>0$ and $c_u>0$, we need $\varepsilon_{1/2}+2\varepsilon_1>0$. This algebraic
condition, which does not depend on $M_{1/2},M_{1}$, immediately rules out the patterns
$(\varepsilon_{1/2},\varepsilon_1)$ equal to $(1,-1)$ or to $(-1,-1)$, leaving exactly the two
admissible configurations $\epsilon^{(1)}=(\varepsilon^{(1)}_{1/2},\varepsilon^{(1)}_1)=(1,1)$ and $\epsilon^{(2)}=(\varepsilon^{(2)}_{1/2},\varepsilon^{(2)}_1)=(-1,1)$. Also observe that, as shown in the proof of the
concentration breaking theorem, $\epsilon^{(1)}$ is the only admissible pattern under the assumptions of
Theorem \ref{teo_breaking}. Then simple calculations give
\[
c_1=\frac{|\Omega|}{3\pi}=\frac14, \qquad
c_2=\frac{|\Omega|}{\pi}=\frac34,
\]
where $c_1$ and $c_2$ are, respectively, the constants $c_u$ referred to $\epsilon^{(1)}$ and to
$\epsilon^{(2)}$. 

\smallskip
\textit{Step 2. The two auxiliary radial profiles.}
For $i=1,2$ let $v_i$ be the unique solution of
\[
-\Delta v_i=1 \text{ in }\Omega,\quad
\de_\nu v_i=-c_i\,\varepsilon^{(i)}_{1/2} \text{ on }\Gamma_{1/2},\quad
\de_\nu v_i=-c_i\,\varepsilon^{(i)}_1 \text{ on }\Gamma_1,\quad
\int_{\partial\Omega}v_i\,d\mathcal H^1=0.
\]
The equation on $\Omega$ gives $v_i(r)=-\tfrac{r^2}{4}+C_i\log r+D_i$, and the boundary conditions
\[
\begin{cases}
	\de_\nu v_i\big|_{\Gamma_{1/2}}=-v_i'\!\left(\tfrac12\right)=\tfrac14-2C_i=-c_i\varepsilon^{(i)}_{1/2},\\[6pt]
	\de_\nu v_i\big|_{\Gamma_1}=v_i'(1)=-\tfrac12+C_i=-c_i\varepsilon^{(i)}_1,
\end{cases}
\]
provide $C_1=\tfrac14$ and $C_2=-\tfrac14$; finally $\int_{\partial\Omega}v_i\,d\mathcal H^1=0$, i.e.\
$\pi\,v_i(\tfrac12)+2\pi\,v_i(1)=0$, yields
\[
D_1=\tfrac{3}{16}+\tfrac{1}{12}\log2,\qquad D_2=\tfrac{3}{16}-\tfrac{1}{12}\log2,
\]
so that $v_1(\tfrac12)=\tfrac18-\tfrac{\log2}{6}$, $v_1(1)=-\tfrac1{16}+\tfrac{\log2}{12}$,
$v_2(\tfrac12)=\tfrac18+\tfrac{\log2}{6}$, $v_2(1)=-\tfrac1{16}-\tfrac{\log2}{12}$. As in the proof of the concentration-breaking theorem, $v_1$ is
the function $u_0$ of \eqref{eq:u0}; here $u_0-g$ is constant on $\partial\Omega$ if and only if
$v_1(\tfrac12)-M_{1/2}=v_1(1)-M_{1}$, i.e.\ if and only if
\[
s=s_0:=v_1(\tfrac12)-v_1(1)=\tfrac{3}{16}-\tfrac{\log2}{4}\approx0.0142 .
\]
We assume $s\neq s_0$, so that the non-degeneracy hypothesis of Theorem \ref{teo_breaking} holds and the
phenomenon below is due solely to the disconnectedness of $\partial\Omega$.

\smallskip
\textit{Step 3. The mass-balance constant.}
If $h_u>0$ everywhere with pattern $\epsilon^{(i)}$, then $u$ and $v_i$ solve the same Neumann problem,
so $u=v_i+K_m$ with $K_m:=K_{m,\epsilon^{(i)}}\in\mathbb{R}$. As $\int_{\partial\Omega}h_u=m$ reads
$\int_{\partial\Omega}|u-g|=c_i(m+3\pi)$ and $\operatorname{sign}(u-g)=\epsilon^{(i)}$ on each component
(i.e.\ $|u-g|=\epsilon^{(i)}(v_i+K_m-g)$), one gets
\[
K_{m}=\frac{c_i(m+3\pi)-I_i}{\Pi_i},
\qquad
I_i:=\int_{\partial\Omega}\varepsilon^{(i)}(v_i-g)\,d\mathcal H^1,
\qquad
\Pi_i:=\varepsilon^{(i)}_{1/2}\,\mathcal{H}^1(\Gamma_{1/2})+\varepsilon^{(i)}_1\,\mathcal{H}^1(\Gamma_1),
\]
with $\Pi_1=3\pi$, $\Pi_2=\pi$. Using that $\int_{\partial\Omega}v_i=0$, one obtains
\[
I_1=\pi\big(v_1(\tfrac12)-M_{1/2}\big)+2\pi\big(v_1(1)-M_{1}\big)=-\pi\,(M_{1/2}+2M_{1}),
\]
\[
I_2=-\pi\big(v_2(\tfrac12)-M_{1/2}\big)+2\pi\big(v_2(1)-M_{1}\big)
=\pi\Big(M_{1/2}-2M_{1}-\tfrac14-\tfrac{\log2}{3}\Big).
\]
After a simple calculation this provides
\[
K_{m,\epsilon^{(1)}}=\frac{m}{12\pi}+\frac14+\frac{M_{1/2}+2M_{1}}{3},
\qquad
K_{m,\epsilon^{(2)}}=\frac{3m}{4\pi}+\frac52+\frac{\log2}{3}-M_{1/2}+2M_{1} .
\]

\smallskip
\textit{Step 4. Double transition.}
Now observe that the pattern $\epsilon^{(i)}$ is realized if and only if
\[
\Phi_i(\sigma)+\epsilon^{(i)}\,K_{m,\epsilon^{(i)}}>c_i\qquad\forall\,\sigma\in\partial\Omega,
\qquad \Phi_i(\sigma):=\epsilon^{(i)}\big(v_i(\sigma)-g(\sigma)\big).
\]
Since $K_{m,\epsilon^{(i)}}$ is increasing in $m$, a component with $\varepsilon^{(i)}=1$ gives a lower
bound for $m$, while a component with $\varepsilon^{(i)}=-1$ gives an upper bound. Hence the uniform
pattern $\epsilon^{(1)}=(1,1)$ produces only lower bounds, whereas the mixed pattern
$\epsilon^{(2)}=(-1,1)$ produces one lower and one upper bound.

For $\epsilon^{(1)}=(1,1)$ we have $\Phi_1=v_1-g$, and the two lower bounds (on $\Gamma_{1/2}$ and on
$\Gamma_1$) read $K_{m,\epsilon^{(1)}}>\tfrac18+\tfrac{\log2}{6}+M_{1/2}$ and
$K_{m,\epsilon^{(1)}}>\tfrac{5}{16}-\tfrac{\log2}{12}+M_{1}$. 
In particular it holds that 
$$(\tfrac18+\tfrac{\log2}{6}+M_{1/2}) - (\tfrac{5}{16}-\tfrac{\log2}{12}+M_{1}) =s-s_0,$$ 
and, for $s>s_0$, the threshold is given by the calculation found for $\Gamma_{1/2}$ and
\[
K_{m,\epsilon^{(1)}}>c_1-\big(v_1(\tfrac12)-M_{1/2}\big)
\quad\Longleftrightarrow\quad
m>\overline m_1:=8\pi s-\frac{3\pi}{2}+2\pi\log2 = 8\pi (s-s_0),
\]
and $\overline m_1$ is the threshold
$\overline m$ given by Theorem \ref{teo_breaking}.

In this disconnected case the pattern $\epsilon^{(2)}=(-1,1)$ is also available. On $\Gamma_{1/2}$,
where $\varepsilon^{(2)}_{1/2}=-1$,
\[
K_{m,\epsilon^{(2)}}<M_{1/2}-v_2(\tfrac12)-c_2
\quad\Longleftrightarrow\quad
m<b_2:=\frac{8\pi}{3}\Big(s-\tfrac{27}{16}-\tfrac{\log2}{4}\Big),
\]
while on $\Gamma_1$, where $\varepsilon^{(2)}_1=1$,
\[
K_{m,\epsilon^{(2)}}>c_2-\big(v_2(1)-M_{1}\big)
\quad\Longleftrightarrow\quad
m>a_2:=\frac{4\pi}{3}\Big(s-\tfrac{27}{16}-\tfrac{\log2}{4}\Big)=\tfrac12\,b_2 .
\]
Hence $\epsilon^{(2)}$ is realized if and only if $m\in(a_2,b_2)$, which is nonempty when the
jump is large enough, that is
\[
s>s^\ast:=\tfrac{27}{16}+\tfrac{\log2}{4}\approx1.8608.
\]
Also observe that, in this case and by a simple calculation, $\overline{m_1}> b_2$. 

\smallskip

Let us finally observe that a standard calculation also gives that the constructed pair $(v_i+K_{m,\epsilon^{(i)}},h_u)$ coincides with the unique global minimizer of Theorem \ref{teo_min}.
\triang
\end{example}
\begin{remark}
	Some comments are in order regarding the previous example. When $s>s^\ast$ one has $0<a_2<b_2<\overline m_1$, and the optimal layer behaves as
	\[
	\underbrace{[0,a_2]}_{h_u=0\text{ somewhere}}\ \cup\
	\underbrace{(a_2,b_2)}_{h_u>0\text{ everywhere}}\ \cup\
	\underbrace{[b_2,\overline m_1]}_{h_u=0\text{ somewhere}}\ \cup\
	\underbrace{(\overline m_1,\infty).}_{h_u>0\text{ everywhere}}
	\]
	Then the previous example, which gives a sort of double transition, shows that Theorem \ref{teo_breaking} is actually optimal. In particular, observe that the half-line $(\overline{m}_1, \infty)$ corresponds to the regime governed by the pattern $\epsilon^{(1)} = (1,1)$ associated with the auxiliary function $u_0$; this is actually in accordance with Theorem \ref{teo_breaking}. Conversely, the extra interval $(a_2, b_2)$ arises from the mixed pattern $\epsilon^{(2)} = (-1,1)$ which is the sign configuration that the hypotheses '$\partial\Omega$ connected' or '$g$ constant' force to exclude.
		
	\medskip
	
	Let us also highlight that, in the previous example, two critical values appear
	\[
	s_0:=\tfrac{3}{16}-\tfrac{\log2}{4}\approx0.0142,\qquad
	s^\ast:=\tfrac{27}{16}+\tfrac{\log2}{4}\approx1.8608,
	\]
	which represent the critical sizes of the temperature jump for $g$ across $\Gamma_{1/2}$ and $\Gamma_1$.
	Let us also stress that the calculations in the previous example are made explicit for $s>s_0$; when $s<s_0$ the roles of the two components are exchanged: the binding
	constraint becomes the one on $\Gamma_1$, and the very same computation performed
	above gives the threshold
	\[
	\overline{m}_1' := 4\pi (s_0-s).
	\]
	Equivalently, for any $s$, the threshold 
	is
	\[
	\overline{m}=\max\bigl\{\,8\pi (s-s_0),\ 4\pi (s_0-s)\,\bigr\},
	\]
	the maximum selecting the component of $\partial\Omega$ which is left uncovered first. 
	Then let us summarize the behaviour of the optimal layer
	\begin{itemize}
		\item[$(a)$] \emph{$s>s^\ast$.} Here $0<a_2<b_2<\overline m_1$ and one has four regimes with
		the double transition described above: $h_u=0$ on one component for $m\in[0,a_2]\cup[b_2,\overline m_1]$
		and $h_u>0$ on all of $\partial\Omega$ for $m\in(a_2,b_2)\cup(\overline m_1,\infty)$.
		\item[$(b)$] \emph{$s_0<s\le s^\ast$.} Now $a_2,b_2\le0$ and $(a_2,b_2)$ disappears. The
		single threshold $\overline m_1$ remains, with $h_u=0$ on $\Gamma_{1/2}$
		for $m\le \overline m_1$ and $h_u>0$ everywhere for $m>\overline m_1$. This is exactly the behaviour of
		Theorem \ref{teo_breaking}.
		\item[$(c)$] \emph{$s=s_0$.} Here $u_0-g$ is constant on $\partial\Omega$, the
		non-degeneracy hypothesis fails, $\overline m_1=0$, and $h_u>0$ everywhere for every $m>0$: concentration
		breaking never occurs.
		\item[$(d)$] \emph{$s<s_0$.} Here the first part of boundary to break becomes $\Gamma_1$; one still has a single
		threshold $\overline{m}_1'$, with $h_u=0$ on $\Gamma_1$
		for $m\le \overline{m}_1'$ and $h_u>0$ everywhere for $m>\overline{m}_1'$. 
	\end{itemize}
	Thus the failure of the single-threshold theorem occurs if and only if the jump's size is large enough, that is $s>s^\ast$. 
	
	We conclude with a summary table.
	
	\begin{center}\small
		\begin{tabular}{lll}
			\hline
			jump $s=M_{1/2}-M_{1}$ & structure & behaviour of $h_u$\\
			\hline
			$s>s^\ast\approx1.861$ & 4 regimes & double transition \\
			$s_0<s\le s^\ast$ & 2 regimes & single threshold $\overline m_1$ (no isolation on $\Gamma_{1/2}$)\\
			$s=s_0\approx0.014$ & --- & $u_0-g$ constant, $\overline m_1=0$, no breaking\\
			$s<s_0$ & 2 regimes & single threshold $\overline{m}_1'$ (no isolation on $\Gamma_1$)\\
			\hline
		\end{tabular}
	\end{center}
	\triang
\end{remark}

\appendix
 \section{Appendix}
In this Appendix, we collect a few auxiliary inequalities that are used throughout the paper.
\begin{lemma}
	Let $a, b\in \mathbb{R}^N$, let $p > 1$ and let $\lambda,\eta > 0$. Then it holds
	\begin{equation}\label{lem_dis}
		\frac{|a-b|^p}{\lambda^{p-1}} + \eta |a|^p \ge \frac{\eta |b|^p}{(1 + \eta^{\frac{1}{p-1}} \lambda)^{p-1}}.
	\end{equation}
\end{lemma}

\begin{proof}
	We first prove that, for any $x, y \ge 0$, $p > 1$, and $\gamma \in (0,1)$, it holds
	\begin{equation} \label{aux}
		(x+y)^p \le  \frac{x^p}{\gamma^{p-1}} + \frac{y^p}{(1-\gamma)^{p-1}},
	\end{equation}
	which is immediate if either  $x=0$ or $y=0$.
	Then assume $x,y > 0$ and observe that in this case inequality \eqref{aux} is equivalent to showing 
	\begin{equation*}
		(1+t)^p \le \frac{1}{\gamma^{p-1}} + \frac{t^{p}}{(1-\gamma)^{p-1}},
	\end{equation*}
	which holds true as it can be easily shown that $g(t) = \displaystyle  \frac{1}{\gamma^{p-1}} + \frac{t^{p}}{(1-\gamma)^{p-1}} - (1+t)^p\ge 0$ for  $t>0$ and it admits a global minimum in $t=\frac{1}{\gamma}-1$.

	Now, in order to prove \eqref{lem_dis}, observe that the triangle inequality yields to
	\begin{equation*} 
		|b|^p \le (|a| + |a-b|)^p.
	\end{equation*}
	Then let us apply \eqref{aux} with $x = |a|$ and $y = |a-b|$ gathering the result into the right-hand of the previous, obtaining
	\begin{equation*}
	|b|^p \le (1 + \eta^\frac{1}{p-1} \lambda)^{p-1} |a|^p + \left(1 + \frac{1}{\eta^\frac{1}{p-1} \lambda}\right)^{p-1} |a-b|^p,
\end{equation*}
	by fixing $\frac{1}{\gamma} =1+ \eta^\frac{1}{p-1} \lambda$. Then, after some simple calculations, the previous inequality gives \eqref{lem_dis}. This concludes the proof.
\end{proof}

We provide a second lemma.
\begin{lemma}
	\label{lemmaconvex}
	If $\varphi(t)$ is a convex function on $\mathbb{R}$, then the function $\displaystyle f(x,y) = y \varphi\left(\frac{x}{y}\right)$ satisfies
	\begin{equation}\label{disconvex}
	\frac{1}{2} f(x_1, y_1) + \frac{1}{2} f(x_2, y_2) \ge f\left(\frac{x_1+x_2}{2}, \frac{y_1+y_2}{2}\right),\qquad \forall x_1,x_2\in \mathbb R,\;\forall y_1,y_2>0.
	\end{equation}
	 Moreover, if $\varphi$ is strictly convex, equality holds if and only if there exists $\lambda > 0$ such that $(x_1, y_1) = \lambda(x_2, y_2)$.
\end{lemma}
\begin{proof}
	Let us first assume that $\varphi$ is convex and let us prove \eqref{disconvex}; namely we have to prove that
	\begin{equation}\label{eq:goal}
		\frac{1}{2} y_1 \varphi\left(\frac{x_1}{y_1}\right) + \frac{1}{2} y_2 \varphi\left(\frac{x_2}{y_2}\right) \ge \left(\frac{y_1+y_2}{2}\right) \varphi\left(\frac{x_1+x_2}{y_1+y_2}\right).
	\end{equation}
	First of all, we can rewrite the argument of $\varphi$ on the right-hand side as
	\begin{equation*}
		\frac{x_1+x_2}{y_1+y_2} = \frac{y_1}{y_1+y_2}\left(\frac{x_1}{y_1}\right) + \frac{y_2}{y_1+y_2}\left(\frac{x_2}{y_2}\right).
	\end{equation*}
	Then fixing $t = \frac{y_1}{y_1+y_2}$, $z_1 = \frac{x_1}{y_1}$ and $z_2 = \frac{x_2}{y_2}$, since $\varphi$ is a convex function, we have
	\[
	\varphi(t z_1 + (1-t) z_2) \le t \varphi(z_1) + (1-t) \varphi(z_2)
	\]
	which substituting back the definitions of $t, z_1, z_2$ and multiplying both sides by $\frac{y_1+y_2}{2}$, gives \eqref{eq:goal}.
	
	\medskip
	
	Now let us assume that $\varphi$ is strictly convex.
	We need to prove that equality in \eqref{eq:goal} holds if and only if the vectors $(x_1, y_1) =\lambda(x_2, y_2)$ for some $\lambda>0$.

	We first assume there exists a constant $\lambda > 0$ such that $(x_1, y_1) = \lambda(x_2, y_2)$ and we prove by a simple calculation that equality needs to occur.
	Indeed the left-hand side of \eqref{eq:goal} gives
	\begin{align*}
		\frac{1}{2} y_1 \varphi\left(\frac{x_1}{y_1}\right) + \frac{1}{2} y_2 \varphi\left(\frac{x_2}{y_2}\right)
		= \frac{1}{2} \lambda y_2 \varphi\left(\frac{x_2}{y_2}\right) + \frac{1}{2} y_2 \varphi\left(\frac{x_2}{y_2}\right) 
		= \frac{1+\lambda}{2} y_2 \varphi\left(\frac{x_2}{y_2}\right)
	\end{align*}
	while the right-hand side of \eqref{eq:goal} yields
	\begin{align*}
		\left(\frac{y_1+y_2}{2}\right) \varphi\left(\frac{x_1+x_2}{y_1+y_2}\right) 
		= \left(\frac{(1+\lambda)y_2}{2}\right) \varphi\left(\frac{(1+\lambda)x_2}{(1+\lambda)y_2}\right) 
		= \frac{1+\lambda}{2} y_2 \varphi\left(\frac{x_2}{y_2}\right).
	\end{align*}

	Conversely, assume that equality holds in \eqref{eq:goal}. Then, as already done in the first part of the proof, we can rewrite the equality as 
	\[
	\varphi\left(\frac{y_1}{y_1+y_2} \frac{x_1}{y_1} + \frac{y_2}{y_1+y_2} \frac{x_2}{y_2}\right) = \frac{y_1}{y_1+y_2} \varphi\left(\frac{x_1}{y_1}\right) + \frac{y_2}{y_1+y_2} \varphi\left(\frac{x_2}{y_2}\right).
	\]
	In particular, as $\varphi$ is strictly convex, the previous equality holds if and only if 
	\[
	\frac{x_1}{y_1} = \frac{x_2}{y_2},
	\]
	which implies that $(x_1, y_1) = \lambda (x_2, y_2)$ for some $\lambda = \frac{y_1}{y_2} > 0$.
	This concludes the proof.
\end{proof}

\section*{Acknowledgements}
The authors wish to thank the GNAMPA-INdAM for partial support of the present work. The authors also thank the anonymous referee for his/her careful reading of the paper, which helped to improve the manuscript.

\end{document}